\documentclass[11pt, twoside, leqno]{amsart}  
\usepackage{lipsum}
\usepackage{amsfonts}
\usepackage{graphicx}
\usepackage{epstopdf}
\usepackage{algorithmic}  
\usepackage{calligra}
\usepackage[dvipsnames]{xcolor}
\usepackage{amsfonts,amsmath,amsthm,amssymb}
\usepackage{mathtools}
\usepackage{hyperref}
\usepackage[makeroom]{cancel}
\usepackage{autonum}
\usepackage{hhline}
\usepackage{array}
\usepackage{diagbox}
\usepackage{mdframed}
\usepackage{multicol}
\usepackage{graphicx}
\usepackage{subcaption}
\usepackage{moreverb}
\usepackage{bbm}
\usepackage[margin=1.38in]{geometry}
\usepackage{todonotes}
\usepackage{scalerel,amssymb}
\allowdisplaybreaks
\usepackage{mathrsfs}  
\usepackage{lineno}
\usepackage{todonotes}
\usepackage{tikz}
\usepackage{appendix}
\usepackage{enumitem}
\usepackage{pgfplots}
\usetikzlibrary{arrows.meta}
\usepackage[numbers,sort&compress]{natbib}
\definecolor{mygreen}{HTML}{43a047}
\usepackage{subcaption}
\usepackage{doi}

\newcommand{\ddt}{\frac{\textup{d}}{\textup{d}t}}

\newcommand{\dt}{\, \textup{d} t}
\newcommand{\ds}{\, \textup{d} s }
\newcommand{\dx}{\, \textup{d} x}

\newcommand{\R}{\mathbb{R}}
\newcommand{\Z}{\mathbb{Z}}

\hypersetup{hidelinks}

\newtheorem{theorem}{Theorem}
\newtheorem{lemma}{Lemma}
\newtheorem{proposition}{Proposition}

\newtheorem{definition}{Definition}
\newtheorem{remark}{Remark}
\numberwithin{lemma}{section}
\numberwithin{proposition}{section}
\numberwithin{theorem}{section}
\numberwithin{equation}{section}
\makeatletter
\newcommand{\leqnomode}{\tagsleft@true}
\newcommand{\reqnomode}{\tagsleft@false}
\makeatother

\definecolor{grey}{rgb}{0.5,0.5,0.5}
        
\title[Nonlocal Hamer system]{On the decay estimates of a nonlocal convection-diffusion Hamer system}        
     
            
\author[Timothée Crin-Barat]{Timoth\'ee Crin-Barat$^\dagger\,^*$}  
\thanks{$^\dagger$Institut de Mathématiques de Toulouse, Université de Toulouse, 118 Route de Narbonne, Toulouse,
31062, France (\href{timothee.crin-barat@math.univ-toulouse.fr}{timothee.crin-barat@math.univ-toulouse.fr})\\$^*$Corresponding author: timothee.crin-barat@math.univ-toulouse.fr}   
\author[Belkacem Said-Houari]{Belkacem Said-Houari$^\ddag$}
\thanks{$^\ddag$Department of Mathematics, College of Sciences, University of
	Sharjah, P.\ O.\ Box: 27272, Sharjah, United Arab Emirates (\href{bhouari@sharjah.ac.ae}{bhouari@sharjah.ac.ae})}

\begin{document}
	\vspace*{8mm}
	\begin{abstract}
We consider the multi-dimensional Hamer model for radiating gases in its coupled hyperbolic--elliptic formulation. By means of energy estimates, we establish the global well-posedness for small initial data in hybrid Besov spaces with distinct regularity exponents at low and high frequencies. This framework enables us to relax the regularity assumptions required in \cite{Duan_Klem_Zhu_2010,Duan_Ruan_Zhu_2012}.
In addition, we establish optimal time-decay estimates for solutions with initial data in the critical Besov space $\dot{B}_{2,\infty}^{-d/2}(\mathbb{R}^d)$, thus extending previous results obtained under the stronger assumption $L^1(\mathbb{R}^d)$. We discuss the optimality of these decay rates and derive improved decay rates under a zero-mass cancellation condition, corresponding to initial data in the larger negative Besov space $\dot B^{-d/2-1}_{2,\infty}(\mathbb R^d)$.
					\end{abstract}   
	\vspace*{-7mm}  
	\maketitle                 
     
\section{Introduction} 
\subsection{Presentation of the system and motivations}
The fundamental system of equations describing the motion of gas in the
presence of radiation is given (see \cite{Ha71,VINCENTI}) by:
\begin{equation}
\left\{ 
\begin{array}{l}
\varrho _{t}+\mathrm{div}(\varrho u)=0,\vspace{0.2cm} \\ 
(\varrho u)_{t}+\mathrm{div}\left( \varrho u\otimes u+p\mathbb{I}\right) =0,%
\vspace{0.2cm} \\ 
\left\{ \varrho (e+\frac{\left\vert u\right\vert ^{2}}{2})\right\} _{t}+%
\mathrm{div}\{\left[\varrho \left(e+\frac{\left\vert u\right\vert ^{2}}{2}\right)+p\right]u+q\}=0,%
\vspace{0.2cm} \\ 
-\nabla \mathrm{div}q+aq+b\nabla \theta ^{4}=0.%
\end{array}
\right.   \label{Radiating_Gas_Main_Equations}
\end{equation}
Here $u$ is the velocity of the gas, $\varrho $ is its density, $p$ is the
pressure, $e$ is the internal energy, $\theta $ is the absolute temperature
of the gas, $q$ is the radiative heat flux and $a$ and $b$ are given
positive constants related to the absorption coefficient $\widetilde{\alpha}$
and the Stefan--Boltzmann constant $\widetilde{\sigma}$ as: $a=3\widetilde{\alpha}%
^{2},b=4\widetilde{\alpha}\widetilde{\sigma}$. The first three equations in  \eqref{Radiating_Gas_Main_Equations}  correspond to the compressible Euler system describing the motion of an inviscid flow. More precisely,    the first equation 
 represents the equation of conservation
of mass (continuity equation), the second equation is the equation of
conservation of momentum (Newton's second law) and  the third equation is the
equation of conservation of energy (first law of thermodynamics). The
fourth equation  represents  the radiative transfer equation (see \cite%
{VINCENTI} and \cite{Ha71} for the derivation of this equation). We assume that the 
pressure $p$ in \eqref{Radiating_Gas_Main_Equations} is given by (the equation of state for a perfect gas):
\begin{equation}
p=\varrho R\theta =A\varrho ^{\gamma }\exp \frac{(\gamma -1)s}{R},
\end{equation}%
where $R$ is the gas constant per unit mass, $s$ is the entropy and $A$ is a positive constant.
 System \eqref{Radiating_Gas_Main_Equations} is a coupling between the Euler equation of a perfect compressible fluid  and an elliptic equation. As such, its mathematical analysis is rather difficult. 

\medbreak
In many mathematical studies, simplified versions of \eqref{Radiating_Gas_Main_Equations} are considered. One such reduction leads to the so-called Hamer model of a radiating gas. More precisely, if we restrict ourselves to the one-dimensional case  and expand the functions $\varrho ,u,s,q$ around the equilibrium state and
retaining the first-order approximation (see \cite{Ha71} and \cite%
{Kawashima_1998}) we end up with the following simplified model, known as
the Hamer model:%
\begin{equation}
\left\{ 
\begin{array}{l}
u_{t}+(u^{2}/2)_{x}+q_{x}=0,\vspace{0.2cm} \\ 
-q_{xx}+q+u_{x}=0. %
\end{array}%
\right.   \label{Hamer_Model}
\end{equation}
The approximating system \eqref{Hamer_Model} takes the form of a hyperbolic-elliptic coupled system and has been studied by many authors. More precisely, the stability of shock waves has been considered by   Kawashima and Nishibata \cite{Kawashima_1998, Kawash_Nishibata_1999}, Lattanzio and Marcati \cite{Lattanzio_2003}, Lattanzio \emph{et al.} \cite{Lattanzio_2009,Lattanzio_2007}, Lin \emph{et al.} \cite{Lin_2007}. See also the work of Iguchi and Kawashima \cite{Iguchi_2002} for diffusion waves.    Serre \cite{Serre_2003} proved the $L^1$-stability for a general flux  $f(u)$ instead of the Burgers flux $u^2/2$ under the assumption of a zero-mass initial disturbance.  He also obtained the $L^1$-stability for shock waves. In fact, he proved that  the limit 
\begin{equation}
\lim_{t\rightarrow +\infty} \Vert u(t)\Vert_{L^1}=0,
\end{equation}
holds, and as a result he showed the decay $t^{-1/2}$ of the $L^2$-norm of the solution, but no decay rate of the $L^1$-norm was provided in \cite{Serre_2003}. See also the paper of Laurençot \cite{Lauren_2005} where the improved decay rate $t^{-3/4}$ of the $L^2$ norm was proved for initial data in the weighted space $u_0\in L^1(\R, (1+|x|)\dx)$ with $\int_\R u_0(x)\dx=0$. 

\medbreak
In this paper, we consider the $d$-dimensional version of the Hamer model for a general flux $f(u)$, namely, we investigate  the following hyperbolic-elliptic coupled system:
\begin{subequations}\label{System_1_Main}
\begin{equation}  \label{System_1_Main_1}
\left\{ 
\begin{array}{c}
u_{t}+\mathrm{div}\left( f\left( u\right) +q\right) =S, \vspace{0.2cm} \\ 
-\nabla \left( \mathrm{div}q-u\right) +q=0,%
\end{array}%
\right.
\end{equation}
supplemented with the initial data 
\begin{eqnarray}
      \label{Initial_data}
u(t=0)=u_0.
\end{eqnarray}
\end{subequations}
Here $u=u\left( x,t\right) :\mathbb{R}^{d}\times \mathbb{R}_{+}\rightarrow 
\mathbb{R}$ and $q\left( x,t\right) :\mathbb{R}^{d}\times \mathbb{R}%
_{+}\rightarrow \mathbb{R}^{d}$ are the unknown functions.  The function $%
f(u)=(f_{1}(u),f_{2}(u),\dots ,f_{d}(u)):\mathbb{R}\rightarrow \mathbb{R}%
^{d} $ and $S=S\left( x,t\right) :\mathbb{R}^{d}\times \mathbb{R}%
_{+}\rightarrow \mathbb{R}$ are given source terms. We assume that the function 
$f\left( u\right) $ is a smooth function of  $u.$    
System \eqref{System_1_Main_1} was derived from \eqref{Radiating_Gas_Main_Equations} in \cite{VINCENTI} through an appropriate approximation. 

 System \eqref{System_1_Main} can be rewritten as a single equation with a
nonlocal operator. Indeed, taking the divergence of the second equation in %
\eqref{System_1_Main_1}, we get%
\begin{equation}
\mathrm{div}q=-\Delta Pu,  \label{q_formula}
\end{equation}%
where the operator $\Delta Pu$ is a pseudo-differential operator defined
via the Fourier transform by the formula 
\begin{equation}
\widehat{Pu}\left( \xi \right) =\frac{1}{1+|\xi |^{2}}\hat{u}\left( \xi
\right) .  \label{P_s_definition}
\end{equation}%
Hence, inserting \eqref{q_formula} into the first equation of \eqref{System_1_Main_1}%
, we obtain by taking (for simplicity) $S=0$, 
\begin{equation}
\left\{ 
\begin{array}{ll}
u_{t}-\Delta Pu+\nabla \cdot f(u)=0, & x\in \mathbb{R}^{d},\,t>0,\vspace{%
0.2cm} \\ 
u\left( x,0\right) =u_{0}\left( x\right) , & x\in \mathbb{R}^{d}.%
\end{array}%
\right.  \label{Main_problem}
\end{equation}
As in \cite{Duan_Ruan_Zhu_2012}, we assume without loss of generality that 
\begin{equation}  \label{Assumtion_f}
f_j(0)=f_j^\prime(0)=0,\qquad 1\leq j\leq d.
\end{equation}
Otherwise, one can take the change of variables 
\begin{eqnarray*}
\widetilde{t}=t,\qquad \widetilde{x}_j=x_j-f^\prime_j(0)t
\end{eqnarray*}
and denote $\widetilde{f}(\cdot)$ by 
\begin{equation}
\widetilde{f}(u)=f(u)-f(0)-f^\prime(0)u,
\end{equation}
so that the form of \eqref{Main_problem} remains unchanged, but %
\eqref{Assumtion_f} still holds for $\widetilde{f}$.

The one-dimensional version of system \eqref{Main_problem} was proposed in 
\cite{Rosenau_1989} as a regularized version of the
Chapman--Enskog expansion for hydrodynamics, which is a perturbed method
based on a power series expansion in terms of a small parameter.

When $P$ is the identity operator and $f(u)=b|u|^{p-1} u$ where $b\in \R^d$ is a constant vector and $p\geq 1$, system \eqref{Main_problem} reduces to the convection-diffusion system 
\begin{equation}
\left\{
\begin{array}{ll}
u_{t}-\Delta u+b\cdot \nabla ( u\left\vert u\right\vert ^{p-1}) =0, & x\in
\mathbb{R}^{d},\,t>0,\vspace{0.2cm} \\
u\left( x,0\right) =u_{0}\left( x\right) , & x\in \mathbb{R}^{d},%
\end{array}%
\right.  \label{Covection_diffusion_Equation}
\end{equation}%
The large time behavior of solutions of \eqref{Covection_diffusion_Equation} has been investigated by many authors. See for
instance \cite{Capiro_1996_2}, \cite{Capiro_1996_1},  \cite{Duro_Carpio_2001}%
, \cite{Duro_Zuazua_1999}, \cite{Escob_Vazq_Zuazua_1993_2}, \cite%
{EscobVasqZuazua_1993}, \cite{EscoZuazua_1991}, \cite{Karch_Schon_2002},
\cite{Zuazua_1993}. Assuming $u_0\in L^1(\mathbb{R}^d )$, then the asymptotic
behavior of the solution of (\ref{Covection_diffusion_Equation}) depends on
the exponent $p$ and we have three different cases: $1<p<1+1/d$, $p=1+1/d$
and $p>1+1/d$. 
These three different cases can be seen by taking the
rescaling function
\begin{equation}
u_\lambda(x,t)=\lambda^d u(\lambda x,\lambda^2 t),
\end{equation}
which satisfies
\begin{equation}  \label{Rescaling_Equation}
\partial_tu_{\lambda}-\Delta u_{\lambda}+\lambda^{d(1-p)+1} b\cdot \nabla (
u_\lambda\left\vert u_\lambda \right\vert ^{p-1}) =0.
\end{equation}
This indicates the following:

\begin{itemize}
\item For $p>1+1/d$, the last term on the left-hand side of (\ref%
{Rescaling_Equation}) vanishes and the diffusion term dominates.
Consequently the solution behaves like the heat kernel, i.e.,
\begin{equation}  \label{heat_Kernel_Behavior}
t^{d/2(1-1/q)}\Vert u(.,t)-MG(.,t)\Vert_{L^q}\rightarrow 0, \quad \text{ as }\quad 
t\rightarrow \infty,
\end{equation}
where 
\begin{equation}
M=\int_{\mathbb{R}^d }u_0(x)dx=\int_{\mathbb{R}^d }u(x,t)dx,\quad \text{and}\quad 
G(x,t)=(4\pi t)^{-d/2}\exp(-|x|^2/(4t))
\end{equation}
is the fundamental solution of the
heat equation.

\item For $p=1+1/d$, the diffusion and the convection are balanced and
we see that $u_\lambda(x,t)$ is also a solution of (\ref%
{Covection_diffusion_Equation}). Consequently, the large time behavior of
the solution is described in terms of self--similar solutions
\begin{equation}
U_M(x,t)=t^{-d/2}U_M(xt^{-1/2},1)
\end{equation}
of equation (\ref{Covection_diffusion_Equation}) with $u_0=M\delta_0$, where $\delta_0$ is the Dirac delta function at zero. 

\item For $1<p<1+1/d$ and for large time, the effect of diffusion is
negligible as compared to convection in the direction of $b$. Therefore, the
asymptotic behavior of the solution to (\ref{Covection_diffusion_Equation})
is given by the fundamental entropy solutions of the reduced equation
\begin{equation}
u_{t}-\Delta^\prime u+b\cdot \nabla ( u\left\vert u\right\vert ^{p-1}) =0,
\end{equation}
where $\Delta^{\prime}$ denotes the $(d-1)$-dimensional Laplacian in the
hyperplane orthogonal to the vector  $b$.
\end{itemize}
A more general version of \eqref{Covection_diffusion_Equation} has been considered in \cite{Duan_Ruan_Zhu_2012} where the operator $P$ has been replaced by the operator $P_s$ whose Fourier symbol is given by  $\frac{1}{(m(\xi))^s}$, with $m(\xi)\approx 1+|\xi|^2$ where the authors proved several decay rates exhibiting a regularity-gain property for $s< 1$ and regularity-loss phenomenon for
$s>1$.

We define the operator $\sqrt{P}$ by its Fourier symbol as:

\begin{equation}
\widehat{\sqrt{P}u}\left( \xi \right) =\frac{1}{\sqrt{1+|\xi|^2}}\hat{u}%
\left( \xi \right) .
\end{equation}
The operator  $Pu$ can be represented as a convolution  operator 
\begin{equation}
Pu=K\ast u,
\end{equation}
where $K$ is the Bessel potential defined by the integral formula  
\begin{equation}
K(x)=\frac{1}{(4\pi)^{d/2}}\int_0^\infty s^{-d/2}e^{-s-\frac{|x|^2}{4s}}ds. 
\end{equation}
It is not difficult to see that the above kernel satisfies the  following property (see \cite[Section V.3.1]{Stein_1}): 
\begin{equation}
K(x)=K(|x|)\geq 0,\qquad \Delta K\ast u=-u+K\ast u\quad \text{and}\quad \int_{\mathbb{R}^d} K(x)dx=1
\end{equation}
Hence, we can recast system \eqref{Main_problem} as 
\begin{equation}\label{Nonlocal_form}
u_{t}+\nabla \cdot f(u)=-u+K\ast u.
\end{equation}
The Cauchy problem associated with the one-dimensional version of \eqref{Nonlocal_form} has been investigated in \cite{Lattanzio_2003} where the authors established both global existence and uniqueness of a weak entropy solution of \eqref{Nonlocal_form}. Moreover, the relaxation limit has also been investigated.   
In \cite{Francesco:2007aa}   the author  extended the result in \cite{Lattanzio_2003}  to higher-dimensional settings.

System \eqref{Main_problem} has been investigated in \cite{Duan_Klem_Zhu_2010}, 
where the authors obtained a  global existence result for small initial-data solution in $H^{s}(\R^d)$ for $s\geq 2[d/2]+2$. They also proved decay estimates for $\|u(t)\|_{L^2}$ for initial data in $L^2(\R^d)\cap L^{1}(\R^d)$.  
\subsection{Aims of our work}
  The aim of this work is twofold: 
  \begin{itemize}
 \item  First,  to improve the global existence result in \cite{Duan_Klem_Zhu_2010} by reducing  the regularity  assumption on the initial data. The operator $\Delta P u$ behaves like the usual Laplacian at the low frequencies $|\xi|\leq 1$ and acts as a
damping term at high frequencies $|\xi|\geq 1$. Motivated by this observation and inspired by the approach developed in \cite{CRINBARAT20221,Danchin_EMS,XuXin}, and with the goal of getting optimal estimates, we rely on the use of hybrid Besov spaces, in which different regularity exponents are assigned to the low and high-frequency regimes.    
   In fact, we show that for small initial data in the hybrid Besov space $\dot{{B}}_{2,1}^{\frac{d}{2}-1,L}(\R^d)\cap \dot{{B}}_{2,1}^{\frac{d}{2}+1,H},\, d\geq 2$, (see Section \ref{Hybrid_Besov} for precise definitions of these Besov spaces) system \eqref{Main_problem} has a unique global solution.   This improves the regularity assumptions in \cite{Duan_Klem_Zhu_2010} and \cite{Duan_Ruan_Zhu_2012}. 
   (See Remark \ref{Remark_Compa} below). 
   
   \item   Secondly, we aim to derive optimal decay estimates characterizing the asymptotic behavior of the solution. More precisely, 
      we establish  decay estimates for the linearized model for initial data in $\dot{{B}}_{2,\infty}^{-d/2}(\mathbb{R}^d) $ thereby improving the corresponding results obtained in  \cite{Duan_Klem_Zhu_2010}. This improvement relies on the embedding $L^1(\mathbb{R}%
^d)\hookrightarrow \dot{{B}}_{2,\infty}^{-d/2}(\mathbb{R}^d)$. These decay estimates are then extended to the nonlinear problem, where various types of decay estimates are established.   
We emphasize that the use of  initial data in $\dot{{B}}_{2,\infty}^{-d/2}(\mathbb{R}^d)$ or more generally in the spaces $\dot{{B}}_{2,\infty}^{-s}(\mathbb{R}^d) $ with $s\in (0, d/2]$ was first introduced by Sohinger and Strain \cite{Sohinger_Strain_2014} in the context of the Boltzmann equation, and by Xu and Kawashima \cite{Xu_Kawashima_2015} for partially dissipative hyperbolic systems.
\end{itemize}

\subsection{Outline of the paper}
The paper is organized as follows. In Section \ref{section_prel}, we introduce notation and recall the main tools from Littlewood-Paley theory, including the definition of Besov spaces and some related preliminary lemmas. 
 Section \ref{Linear_Section} is devoted to the study of the linearized model, where improved decay estimates are obtained.  In Section  \ref{Section_Nonl}, we prove a global existence result by employing the energy method in hybrid Besov spaces. In Section \ref{Section_Decay}, we prove the decay rate of the $L^2$-norm of the solution based on a modified Nash-type inequality. Furthermore, we determine the decay rate of the solution in some Besov norms of the solution, provided that the initial data satisfy some additional assumptions. In addition, we show the optimality of the decay rate of the $L^2$ norm of the solution, by deriving a matching lower bound of the $L^2$ norm of the solution. A faster decay rate of the solution for initial data in $\dot{B}^{-d/2-1}_{2,\infty}(\mathbb{R}^d) $ is obtained in Section \ref{Section_Improved_Decay}. In Section \ref{Section_Local_Existence} we establish a local well-posedness   result. 
 Appendix \ref{Appendix_A} collects several technical estimates on Besov spaces, including commutator, product and composition estimates which are used repeatedly in the proofs.



\section{Preliminaries}\label{section_prel}
\subsection{Littlewood-Paley notations}
 In this section, we collect some notions and results
which turn out to be useful in our proof. First, we recall some basic facts
on Littlewood-Paley theory, then we  introduce a class of hybrid Besov spaces that will be useful later in the proof of our main results.

For any positive constants $A$ and $B$, the notation $A\lesssim B$ means
that there exists a positive constant $C>0$ independent of the relevant parameters, such that $A\le CB.$\smallbreak
Let $\mathcal{S}(\mathbb{R}^d)$ be the Schwartz class of rapidly decreasing functions. For a given $f$
in $\mathcal{S}(\mathbb{R}^d)$, we define its Fourier transform $\mathcal{F}%
f=\hat{f}$ and its inverse $\mathcal{F}^{-1}\hat{f}=\check{f}$ as 
\begin{equation}
\hat{f}(\xi)=\int_{\mathbb{R}^d}e^{-ix\cdot \xi} f(x)
dx,\quad \text{and}\quad \check{f}(x)=\frac{1}{(2\pi)^{d}}\int_{\mathbb{R}^d}e^{i x\cdot\xi}\hat{%
f}(\xi)d\xi.
\end{equation}

We denote by 
\begin{equation}
\Lambda^\alpha f=\mathcal{F}^{-1}(|\xi|^\alpha\hat{f}(\xi)),\quad \alpha\in 
\mathbb{R}.
\end{equation}


Throughout the paper, we fix a homogeneous  Littlewood-Paley decomposition $(\dot{\Delta})_{j\in\Z}$
that is defined  by 
$$\dot{\Delta}_j\triangleq\varphi(2^{-j}D) \quad \text{with}\quad\varphi(\xi)\triangleq \chi(\xi/2)-\chi(\xi)$$
where $\chi$ stands for a  smooth function  with range in $[0,1],$ supported in  $B(0,4/3)$ and
such that $\chi\equiv1$ on  $B(0,3/4)$. 
We further set 
$$\dot S_j\triangleq \chi(2^{-j}D) \quad\hbox{for all }\ j\in\Z$$
and define by $\mathcal{S}_h$ the space of
tempered distributions modulo polynomials $\mathcal{P}$.
It is not hard to see that the space $\mathcal{S}_h^\prime$ is exactly the
space of tempered distributions for which we may write 
\begin{equation}
u=\sum_{k\in \mathbb{Z}}\dot{\Delta}_ku.
\end{equation}
This decomposition is called the homogeneous Littlewood-Paley decomposition.

 We have the following  Bernstein inequality. See \cite[Lemma 2.1]{Bahouri_2011_1}.
\begin{lemma}
\label{Lemma_Tao} Let $k$ be an integer and let $f$ be a function with
Fourier support in the annulus $\mathcal{C}=\{2^{k-1}\leq |\xi|\leq
2^{k+1}\} $. Then we have 
\begin{equation}
2^k \Vert f\Vert_{L^p}\lesssim \Vert \nabla f\Vert_{L^p} \lesssim 2^k \Vert
f\Vert_{L^p},
\end{equation}
for all $1\leq p\leq \infty$. In particular, we have 
\begin{equation}
2^k \Vert \dot{\Delta}_kf\Vert_{L^p}\lesssim \Vert \nabla \dot{\Delta}%
_kf\Vert_{L^p} \lesssim 2^k \Vert \dot{\Delta}_kf\Vert_{L^p}.
\end{equation}
\end{lemma}

The following two lemmas are adapted to the linear operator of the system we consider.

We have the following Bernstein-type inequality.
\begin{lemma}
\label{Bernstein_inequality} For any $u\in \mathcal{S}^{\prime }$, there
exists a positive constant $c$ such that 
\begin{equation}
c2^{k}\Vert \dot{\Delta}_{k}u\Vert _{L^{2}}\leq \Vert \nabla\sqrt{P}\dot{\Delta}%
_{k}u\Vert _{L^{2}},\qquad \text{if }k<0,
\end{equation}%
and 
\begin{equation}
c\Vert \dot{\Delta}_{k}u\Vert _{L^{2}}\leq \Vert \nabla\sqrt{P}\dot{\Delta}%
_{k}u\Vert _{L^{2}},\qquad \text{if }k\geq 0.
\end{equation}
\end{lemma}

\begin{proof}
The proof of the above result is a consequence of the Plancherel identity
and the fact that 
\begin{equation}\label{abs}
\frac{|\xi|^2}{1+|\xi|^2}\sim \min\bigl(1,2^{2k}\bigr)
\qquad\text{on } \operatorname{supp}\widehat{\dot{\Delta}_k}.
\end{equation}
\end{proof}

\begin{lemma}
\label{Lemma_heat_Semi_Group} We have the following estimates 
\begin{subequations}
\begin{equation}  \label{Exponential_Type_Estimate_2}
\Vert e^{t\Delta P}\dot{\Delta}_ku\Vert_{L^2}\leq e^{-c\frac{t}{2}}\Vert 
\dot{\Delta}_k u\Vert_{L^2},\quad \text{for }\quad k\geq0,
\end{equation}
and 
\begin{equation}  \label{Exponential_Type_Estimate}
\Vert e^{t\Delta P}\dot{\Delta}_ku\Vert_{L^2}\leq e^{-c t2^{2k}}\Vert 
\dot{\Delta}_k u\Vert_{L^2},\quad \text{for }\quad k<0,
\end{equation}
\end{subequations}
\end{lemma}

\begin{proof}
Using Plancherel theorem and keeping in mind \eqref{P_s_definition}, we have 
\begin{eqnarray*}
\Vert e^{t\Delta P}\dot{\Delta}_ku\Vert_{L^2}&=&\Vert \widehat{e^{-
t\Delta P}\dot{\Delta}_ku}\Vert_{L^2} \\
&=&\Vert e^{- t\eta(\xi)}\widehat{\dot{\Delta}_ku}\Vert_{L^2}
\end{eqnarray*}
with $\eta(\xi)=|\xi|^2/(1+|\xi|^2)$. Now, for $k\geq 0$,using \eqref{abs}, we get 
\begin{eqnarray*}
\Vert e^{- t\eta(\xi)}\widehat{\dot{\Delta}_ku}\Vert_{L^2}&\lesssim&\Vert e^{-\frac{1}{2}t }\widehat{\dot{\Delta}_ku}%
\Vert_{L^2} 
\lesssim e^{-\frac{1}{2}t}\Vert \dot{\Delta}_ku\Vert_{L^2}.
\end{eqnarray*}
This yields \eqref{Exponential_Type_Estimate_2}. For $k<0$, we have $|\xi|\sim 2^k<1$. Hence, we obtain 
\begin{eqnarray*}
\Vert e^{- t\eta(\xi)}\widehat{\dot{\Delta}_ku}\Vert_{L^2}&\leq&\Vert e^{- t\frac{|\xi|^2}{2}}\widehat{\dot{\Delta}_ku}%
\Vert_{L^2} \\
&\leq & Ce^{-c t2^{2k}}\Vert \dot{\Delta}_k u\Vert_{L^2}.
\end{eqnarray*}
This implies \eqref{Exponential_Type_Estimate} and hence 
completes the proof of Lemma \ref{Lemma_heat_Semi_Group}. 
\end{proof}

\subsection{Besov spaces} 
\begin{definition}
For $s\in \mathbb{R}$ and $1\leq p,q\leq \infty$, the homogeneous Besov space 
$\dot{B}_{p,q}^{s}$ is defined as 
\begin{equation}
\dot{B}_{p,q}^{s}=\left\{f\in \mathcal{S}_h^\prime : \Vert f\Vert_{\dot{B}%
_{p,q}^{s}}<\infty\right\},
\end{equation}
where 
\begin{equation}
\Vert f\Vert_{\dot{B}_{p,q}^{s}}=\left\{ 
\begin{array}{ll}
\left(\displaystyle\sum_{j\in \mathbb{Z}} \left(2^{js}\Vert \dot{\Delta}_j
f\Vert_p\right)^q\right)^{1/q}, & \qquad q<\infty\vspace{0.2cm} \\ 
\displaystyle\sup_{j\in\mathbb{Z}}2^{js}\Vert \dot{\Delta}_j f\Vert_p, & 
\qquad q=\infty%
\end{array}
\right.
\end{equation}
\end{definition}



\subsection{Hybrid Besov spaces}\label{Hybrid_Besov}
Since the operator $P$ behaves differently in low and high frequencies, it is suitable to split any tempered distribution $f$ in $\dot{B}_{p,r}^{s}$
into $f=f^L+f^H$, where  
\begin{equation}
f^{L}=\sum_{k\leq 0}\dot{\Delta}_{k}f=\dot S_{1}f\qquad \text{and}\qquad
f^{H}=\sum_{k>0}\dot{\Delta}_{k}f=(I-\dot S_{1})f.
\end{equation}%
Accordingly, we introduce Besov semi-norms restricted to low and high frequencies:
\begin{equation}
\Vert f\Vert _{\dot{B}_{2,1}^{s,L}}=\sum_{k\leq 0}2^{ks}\Vert \dot{\Delta}%
_{k}f\Vert _{L^{2}}\qquad \text{and}\qquad \Vert f\Vert _{\dot{B}%
_{2,1}^{s,H}}=\sum_{k>0}2^{ks}\Vert \dot{\Delta}_{k}f\Vert _{L^{2}}.
\end{equation}%

\begin{definition}[\cite{Danchin_2001_2}]
Let $\sigma$ and $s\in \mathbb{R}$. We define the hybrid Besov space $\dot{B}%
_{2,1}^{s,\sigma}$ by
\begin{eqnarray*}
\dot{B}_{2,1}^{s,\sigma}=\{u\in \mathcal{S}_h^\prime: \Vert u\Vert_{\dot{B}%
_{2,1}^{s,\sigma}}<+\infty\}
\end{eqnarray*}
where 
\begin{eqnarray*}
\Vert u\Vert_{\dot{B}_{2,1}^{s,\sigma}}:=\Vert u\Vert _{\dot{B}_{2,1}^{s,L}}
 +\Vert u\Vert _{\dot{B}_{2,1}^{\sigma,H}}.
\end{eqnarray*}
\end{definition}

In the following, we will repeatedly use the following inequalities: for $s_1,s_2\in\R$ such that $s_1\leq s_2$, one has
\begin{equation}\label{Inequality_h_low}
\Vert u\Vert _{\dot{B}%
_{2,1}^{s_1,L}}\geq \Vert u\Vert _{\dot{B}%
_{2,1}^{s_2,L}}\quad \text{and}\quad \Vert u\Vert _{\dot{B}%
_{2,1}^{s_1,H}}\leq \Vert u\Vert _{\dot{B}%
_{2,1}^{s_2,H}}.
\end{equation}

%
%
%
%
%
%
%
%
%
%

Now, we introduce the following lemma.

\begin{lemma}
\label{Embedding_Lemma} Suppose that $\varrho>0$ and $1\leq p<2$. It holds
that 
\begin{equation}
\Vert f\Vert_{\dot{B}^{-\varrho}_{r,\infty}}\lesssim\Vert f\Vert_{L^p}
\end{equation}
with $1/p-1/r=\varrho/d$. In particular this holds with $\varrho=d/2,\, r=2$
and $p=1$.
\end{lemma}

The above lemma shows in particular the embedding $L^1(\mathbb{R}^d)\hookrightarrow \dot{B}%
^{-\frac d2}_{2,\infty}(\mathbb{R}^d)$. This allows us to replace the classical $L^1$ assumption (commonly used to derive decay estimates) by a weaker condition formulated in terms of  $L^2$-based negative Besov spaces. 
\section{The linear model}\label{Linear_Section}

In this section, we consider the linear model 
\begin{equation}
\left\{ 
\begin{array}{ll}
u_{t}-\Delta Pu=0, & x\in \mathbb{R}^{d},\,t>0,\vspace{0.2cm} \\ 
u\left( x,0\right) =u_{0}\left( x\right) , & x\in \mathbb{R}^{d}%
\end{array}%
\right.  \label{Main_problem_Linear}
\end{equation}
and investigate the time behavior of its solutions in Besov regularity. Our main goal is to establish decay estimates under minimal regularity on the initial data.  We also capture both the low-frequency and high-frequency behavior of the solution, which will play a crucial role later in the analysis of the optimal decay estimate of the nonlinear model \eqref{Main_problem}. 
We have the following result.

\begin{theorem}
\label{Theorem_Main_Linear} Let $\sigma\in \mathbb{R}$ and $s\in \mathbb{R}$
satisfying $\sigma+s>0$. Let $u$ be the  solution of \eqref{Main_problem_Linear} associated with an initial data $u_0\in \dot{B}^{-s}_{2,\infty}\cap \dot{B}%
^{\sigma}_{2,r}$. Then, there exists a constant $C_0$ such that 
\begin{equation}  \label{Main_Inequality_Bes}
\Vert u(t)\Vert_{\dot{B}^{-s}_{2,\infty}}\leq C_0.
\end{equation}
In addition, the following decay estimate holds: 
\begin{equation}  \label{Main_Decay_estimate}
\Vert u(t)\Vert_{\dot{B}^{\sigma}_{2,r}}\leq C (1+t)^{-\frac{\sigma+s}{2}%
}\Vert u_0 \Vert_{\dot{B}_{2,\infty}^{-s}}+Ce^{-\lambda t}\Vert u_0 \Vert_{%
\dot{B}_{2,r}^{\sigma}}, 
\end{equation}
where $\lambda$ and $C$ are positive constants.
\end{theorem}

Before giving the proof of the above theorem, we make several remarks.

\begin{remark} The decay rate of the $L^p$-norms of the solutions can be recovered from \eqref{Main_Decay_estimate}. Indeed, 
using the embeddings 
\begin{equation}
\dot{B}^{\sigma}_{2,1}(\mathbb{R}^d) \hookrightarrow\dot{H}^{\sigma}(\mathbb{%
R}^d)\hookrightarrow L^p(\mathbb{R}^d),\qquad \sigma=d(1/2-1/p), \quad 2\leq
p<\infty,
\end{equation}
we get 
\begin{equation}  \label{Estimate_L_p}
\begin{aligned}
\Vert u(t)\Vert_{L^p}\leq C\Vert u(t)\Vert_{\dot{H}^{\sigma}}&\leq C\Vert
u(t)\Vert_{\dot{B}^{\sigma}_{2,1}}   \\
&\leq C(\Vert u_0 \Vert_{\dot{B}_{2,\infty}^{-s}}+\Vert u_0 \Vert_{\dot{B}%
_{2,1}^{\sigma}}) (1+t)^{-\frac{d}{2}(1/2-1/p)-s/2}.
\end{aligned}
\end{equation}
\end{remark}

\begin{remark}
For $p=2$, we recover from \eqref{Estimate_L_p} the decay estimate 
\begin{equation}  \label{Estimate_L_p_2}
\Vert u(t)\Vert_{L^2} \leq C(\Vert u_0 \Vert_{\dot{B}_{2,\infty}^{-s}}+%
\Vert u_0 \Vert_{\dot{B}_{2,1}^{0}}) (1+t)^{-s/2}.
\end{equation}
For $s=d/2$, we get the same decay estimate $(1+t)^{-d/4}$ of the $L^2$-norm
obtained in \cite{Duan_Klem_Zhu_2010}, but under the assumption $u_0\in \dot{%
{B}}_{2,\infty}^{-d/2}$ which is weaker than the $L^1$-assumption used in 
\cite{Duan_Klem_Zhu_2010} due to the embedding $L^1(\mathbb{R}%
^d)\hookrightarrow \dot{{B}}_{2,\infty}^{-d/2}(\mathbb{R}^d)$.
\end{remark}

\begin{proof}[Proof of Theorem \protect\ref{Theorem_Main_Linear}]

A natural way to get estimates for the solutions of %
\eqref{Main_problem_Linear} in Besov spaces is first to localize the
equation using  the Littlewood--Paley decomposition and to bound each
dyadic block separately in $L^{\rho}(0,T; L^p)$ spaces. The desired bounds 
are subsequently obtained by 
performing a weighted $\ell^r$ summation. Applying the
frequency-localization operator $\dot{\Delta}_k$ to %
\eqref{Main_problem_Linear}, we get 
\begin{equation}
\left\{ 
\begin{array}{ll}
\partial_t\dot{\Delta}_ku-\Delta P\dot{\Delta}_ku=0,\vspace{0.2cm} &  \\ 
\dot{\Delta}_ku\left( x,0\right) =\dot{\Delta}_ku_{0}\left( x\right). & 
\end{array}%
\right.  \label{Main_problem_k}
\end{equation}
Taking the standard inner product with $\dot{\Delta}_ku$, we obtain 
\begin{eqnarray*}
\frac{1}{2}\ddt \Vert \dot{\Delta}_ku\Vert_{L^2}^2+\Vert\nabla \sqrt{P%
}\dot{\Delta}_k u\Vert_{L^2}^2=0.
\end{eqnarray*}
Using Lemma \ref{Bernstein_inequality}, we get  
\begin{eqnarray}  \label{First_k_estimate}
\ddt\Vert \dot{\Delta}_ku\Vert_{L^2}^2+c\min(1, 2^{2k})\Vert \dot{\Delta}_k
u\Vert_{L^2}^2\leq 0.
\end{eqnarray}
The estimate \eqref{First_k_estimate} indicates a gain of two derivatives in low frequencies, as in the heat equation, while in high frequencies it yields exponential decay.

Now, for  $k<0$, multiplying \eqref{First_k_estimate}  by $2^{2k\ell}$, we obtain  
\begin{equation}
\ddt \left(2^{2k\ell}\Vert \dot{\Delta}_ku\Vert_{L^2}^2%
\right)+ c 2^{2k(\ell+1)}\Vert \dot{\Delta}_k u\Vert_{L^2}^2\leq 0.
\end{equation}
Similarly, for $k\geq 0$, we have 
\begin{equation}
\ddt \left(2^{2k\ell}\Vert \dot{\Delta}%
_ku\Vert_{L^2}^2\right)+c 2^{2k\ell} \Vert \dot{\Delta}_k
u\Vert_{L^2}^2\leq 0.
\end{equation}
Collecting the above two estimates, integrating-in-time and taking the supremum over $k\in \mathbb{Z}$ give
\begin{equation}
\Vert u(t)\Vert_{\dot{B}^{\ell}_{2,\infty}}\leq \Vert u_0\Vert_{\dot{B}%
^{\ell}_{2,\infty}}.
\end{equation}
Choosing $\ell=-s$, \eqref{Main_Inequality_Bes} holds.


To show \eqref{Main_Decay_estimate}, applying the Fourier transform to \eqref{Main_problem_k}, we get 
\begin{equation}
\left\{ 
\begin{array}{ll}
\partial_t\widehat{\dot{\Delta}_ku}+\dfrac{|\xi|^2}{1+|\xi|^2}\widehat{\dot{%
\Delta}_ku}=0,\vspace{0.2cm} &  \\ 
\widehat{\dot{\Delta}_ku}\left( \xi,0\right) =\widehat{\dot{\Delta}_ku_{0}}%
\left( \xi\right). & 
\end{array}%
\right.  \label{Main_problem_k_Fourier}
\end{equation}
Solving system \eqref{Main_problem_k_Fourier}, we obtain  
\begin{equation}
\widehat{\dot{\Delta}_ku}(\xi,t)=e^{-\eta(|\xi|)t}\widehat{\dot{\Delta}_ku}%
_0(\xi),\qquad \text{with}\qquad \eta(|\xi|)=\frac{|\xi|^2}{1+|\xi|^2}.
\end{equation}
Hence, we get (see \cite{Kawashima_1}) 
\begin{equation}\label{L_2_Est_Linear}
\left\Vert\widehat{\dot{\Delta}_ku}(\xi,t)\right\Vert_{L^2}=\left\Vert
e^{-\eta(|\xi|)t}\widehat{\dot{\Delta}_ku}_0(\xi)\right\Vert_{L^2}. 
\end{equation}
Using the definition of the localization operator, we have $|\xi|\sim 2^k$. For $k<0$, we have $\eta(\xi)\sim |\xi|^2$. Thus, for $c>0$, 
\begin{eqnarray*}
\left\Vert\widehat{\dot{\Delta}_ku}(\xi,t)\right\Vert_{L^2}^2&=&\left\Vert
e^{-\eta(|\xi|)t}\widehat{\dot{\Delta}_ku}_0(\xi)\right\Vert_{L^2}^2 \\
&=&\int_{\mathbb{R}^d}e^{-2\eta(|\xi|)t}\vert\widehat{\dot{\Delta}_ku}%
_0(\xi) \vert^2d\xi.
\end{eqnarray*}
Since $k<0$, we have
\begin{eqnarray*}
\int_{\mathbb{R}^d}e^{-2\eta(|\xi|)t}\vert\widehat{\dot{\Delta}_ku}_0(\xi)
\vert^2d\xi&\lesssim& \int_{\mathbb{R}^d}e^{-2c|\xi|^2t}\vert\widehat{\dot{%
\Delta}_ku}_0(\xi) \vert^2d\xi \\
&\lesssim& \int_{\mathbb{R}^d}\vert\widehat{\dot{\Delta}_ku}_0(\xi)
\vert^2e^{-2c(2^k\sqrt{t})^2}d\xi.
\end{eqnarray*}
This gives 
\begin{eqnarray*}
\left\Vert\widehat{\dot{\Delta}_ku}(\xi,t)\right\Vert_{L^2(\mathbb{R}%
^d)}&\lesssim&\Vert\widehat{\dot{\Delta}_ku}_0(\xi) e^{-c(2^k\sqrt{t}%
)^2}\Vert_{L^2}.
\end{eqnarray*}

Let $\sigma\in\mathbb{R}$ and $s\in \mathbb{R}$ such that $\sigma+s>0$. Multiplying the above inequality by $2^{k\sigma}$ and taking the $\ell^r(%
\mathbb{Z})$ norm for $1\leq r\leq \infty$, we get 
\begin{eqnarray}  \label{Estimate_1}
\left\Vert 2^{k\sigma}\left\Vert\widehat{\dot{\Delta}_ku}(\xi,t)\right%
\Vert_{L^2}\right\Vert_{\ell^r(\mathbb{Z})} &\lesssim &\Vert u_0 \Vert_{\dot{%
B}_{2,\infty}^{-s}}(1+t)^{-\frac{\sigma+s}{2}} \Vert(2^k\sqrt{t})^{\sigma+s}
e^{-c(2^k\sqrt{t})^2}\Vert_{\ell^r(\mathbb{Z})}  \notag \\
&\lesssim& \Vert u_0 \Vert_{\dot{B}_{2,\infty}^{-s}}(1+t)^{-\frac{\sigma+s}{2%
}}
\end{eqnarray}
where we have used the fact that 
 $\Vert(2^k\sqrt{t})^{\sigma+s} e^{-c(2^k\sqrt{t})^2}\Vert_{\ell^r(%
\mathbb{Z})}\lesssim  1$ and performed the usual $t\leq 1$ and $t\geq 1$ splitting analysis.

Now, for $k\geq 0$, we have $\eta(\xi)\sim c$. Hence, we have 
\begin{eqnarray*}
\int_{\mathbb{R}^d}e^{-2\eta(|\xi|)t}\vert\widehat{\dot{\Delta}_ku}_0(\xi)
\vert^2d\xi&\lesssim& \int_{\mathbb{R}^d}e^{-ct}\vert\widehat{\dot{\Delta}_ku%
}_0(\xi) \vert^2d\xi \\
&\lesssim&e^{-ct}\Vert\dot{\Delta}_ku_0 \Vert_{L^2}^2.
\end{eqnarray*}
This leads to 
\begin{eqnarray*}
\left\Vert\widehat{\dot{\Delta}_ku}(\xi,t)\right\Vert_{L^2(\mathbb{R}^d)}&\lesssim&e^{-\frac{c}{2}t}\Vert\widehat{\dot{\Delta}_ku}_0(\xi)
\Vert_{L^2}.
\end{eqnarray*}
Multiplying the above inequality by $2^{k\sigma}$ and taking the $%
\ell^r(\mathbb{Z})$ norm, for $1\leq r\leq \infty$, we get 
\begin{eqnarray}  \label{Estimate_2}
\left\Vert 2^{k\sigma}\left\Vert\widehat{\dot{\Delta}_ku}(\xi,t)\right%
\Vert_{L^2}\right\Vert_{\ell^r(\mathbb{Z})} &\lesssim &e^{-\frac{c}{2} t}\Vert u_0
\Vert_{\dot{B}_{2,r}^{\sigma}}. 
\end{eqnarray}
Collecting the two estimates \eqref{Estimate_1} and \eqref{Estimate_2}, and using Plancherel's identity, we deduce \eqref{Main_Decay_estimate}.
\end{proof}

\section{Global well-posedness of the nonlinear model}
\label{Section_Nonl}
In this section, we are concerned with the global-in-time well-posedness of the nonlinear system \eqref{Main_problem}. More precisely, we prove the global existence and uniqueness of
a solution at critical regularity.

 First, we should mention that the solution of %
\eqref{Main_problem} satisfies the properties 
\begin{equation}
\Vert u(t)\Vert _{L^{1}}\leq \Vert u_{0}\Vert _{L^{1}},\qquad \Vert
u(t)\Vert _{L^{\infty }}\leq \Vert u_{0}\Vert _{L^{\infty }}
\label{Contra_1}
\end{equation}%
and the conservation of mass property: 
\begin{equation}\label{Mass_Cons}
\int_{\mathbb{R}^{d}}u(x,t)dx=\int_{\mathbb{R}^{d}}u_{0}(x)dx.
\end{equation}

Our main result reads as follows:
\begin{theorem}
\label{Main_Theorem_Nonl} Let $d\geq 2$ and  $u_{0}\in \dot{B}_{2,1}^{\frac{d}{2}-1}(%
\mathbb{R}^{d})\cap \dot{B}_{2,1}^{\frac{d}{2}+1}(%
\mathbb{R}^{d})$. There exists a positive constant $\alpha$ such that if
\begin{equation}\label{alpha_condition}
\Vert u_{0}\Vert _{\dot{{B}}_{2,1}^{\frac{d}{2}-1,L}}+\Vert u_{0}\Vert _{\dot{{B}}_{2,1}^{\frac{d}{2}+1,H}}\leq \alpha,
\label{Initial_Assumption_Samll}
\end{equation}
then \eqref{Main_problem} admits a unique global-in-time solution satisfying
\begin{equation}
u\in C([0,\infty);\dot{B}_{2,1}^{\frac{d}{2}-1}(%
\mathbb{R}^{d})\cap \dot{B}_{2,1}^{\frac{d}{2}+1}(%
\mathbb{R}^{d}))\cap L^{1}((0,\infty);\dot{B}_{2,1}^{\frac{d}{2}+1}(%
\mathbb{R}^{d})).
\end{equation}
\end{theorem}

\begin{remark}
\label{Remark_Compa} A global existence result for small data $u_0\in H^{s}(\mathbb{R}^{d})$ with $s\geq 2[d/2]+2$  was established in \cite%
{Duan_Klem_Zhu_2010}. In contrast, the global existence result stated in Theorem \ref{Main_Theorem_Nonl} is obtained in a hybrid Besov framework requiring only the smallness in the spaces  
\begin{equation}
\dot{B}_{2,1}^{\frac{d}{2}-1}(%
\mathbb{R}^{d})\cap \dot{B}_{2,1}^{\frac{d}{2}+1}(%
\mathbb{R}^{d})
\end{equation} By the embedding 
$
\dot{H}^{2[d/2]+2}\hookrightarrow \dot{{B}}%
_{2,1}^{d/2+1},
$
it follows that the regularity requirement in Theorem \ref{Main_Theorem_Nonl} is weaker than that in \cite{Duan_Klem_Zhu_2010}. 
\end{remark}

\begin{proof}[Proof of Theorem \protect\ref{Main_Theorem_Nonl}]
The proof relies on an energy method in hybrid Besov spaces. More precisely, we first localize 
the equation by means of dyadic decomposition and treat the low-frequencies and high-frequencies regimes separately. This frequency splitting is motivated by the structure of the nonlocal linear operator $\Delta P$, which, at low frequencies, yields a parabolic-type smoothing effect, whereas at high frequencies, it behaves essentially as a linear damping operator. Accordingly, the low-frequency part of the solution is estimated in the Besov spaces  $\dot{{B}}_{2,1}^{\frac{d}{2}-1}$ and the high-frequency part is controlled in $\dot{{B}}_{2,1}^{\frac{d}{2}+1}$ so as to recover $L^1_T(\dot{W}^{1,\infty})$ bounds on the solution. The nonlinearities are handled via suitable commutator estimates. After obtaining the desired a priori estimates, we use a standard continuity argument to prove the global existence. Similar arguments have also been employed in \cite{Crin-Bara_Danchin_2021,CRINBARAT20221,Danchin_EMS} to show global existence for some classes of hyperbolic systems of balance laws.

To  prove the \emph{a priori} bound, 
we define
\begin{equation}
\mathcal{E}(t):=\Vert u(t)\Vert_{L^\infty_t(\dot{{B}}_{2,1}^{\frac{d}{2}-1,L})}+\Vert u(t)\Vert_{L^\infty_t(\dot{{B}}_{2,1}^{\frac{d}{2}+1,H})},\qquad \mathcal{D}(t):= \|u\|_{L_t^1(\dot{{B}}_{2,1}^{\frac{d}{2}+1})}
\end{equation}
 and 
 \begin{equation}
\Vert u(t)\Vert_X:= \mathcal{E}(t)+\mathcal{D}(t).
\end{equation}
We also define
\begin{equation}
\|u_0\|_{X}=\Vert u_{0}\Vert _{%
\dot{B}_{2,1}^{\frac{d}{2}-1,L}}+\Vert
u_{0}\Vert _{\dot{{B}}_{2,1}^{\frac{d}{2}+1 ,H}}.
\end{equation}
and 
\begin{equation}
\|u\|_{\mathcal{X}}=\Vert u\Vert _{%
\dot{B}_{2,1}^{\frac{d}{2}-1,L}}+\Vert
u\Vert _{\dot{{B}}_{2,1}^{\frac{d}{2}+1 ,H}}.
\end{equation}
\textbf{Low-frequency a priori estimates:}
Applying the frequency operator $\dot{\Delta}_{k},\,k\in \mathbb{Z}$ to %
\eqref{Main_problem}, we get 
\begin{equation}
\left\{ 
\begin{array}{ll}
\partial _{t}\dot{\Delta}_{k}u-\Delta P\dot{\Delta}_{k}u+\dot{\Delta}%
_{k}(\nabla \cdot f(u))=0,\vspace{0.2cm} &  \\ 
\dot{\Delta}_{k}u\left( x,0\right) =\dot{\Delta}_{k}u_{0}\left( x\right) . & 
\end{array}%
\right.  \label{Main_N_problem_k}
\end{equation}%
As for \eqref{First_k_estimate}, taking the inner product in $L^{2}$ of the
first equation in \eqref{Main_N_problem_k} with $\dot{\Delta}_{k}u$, we get,
for $k<0$, 
\begin{equation}
\frac{1}{2}\ddt\Vert \dot{\Delta}_{k}u\Vert _{L^{2}}^{2}+c2^{2k}\Vert \dot{%
\Delta}_{k}u\Vert _{L^{2}}^{2}\leq \int_{\mathbb{R}^{d}}\dot{\Delta}%
_{k}(\nabla \cdot f(u))\dot{\Delta}_{k}udx.
\end{equation}
We have  
\begin{eqnarray}  \label{First_k_estimate_Non}
\left\vert \int_{\mathbb{R}^{d}}\dot{\Delta}_{k}(\nabla \cdot f(u))\dot{%
\Delta}_{k}udx\right\vert &\leq &\int_{\mathbb{R}^{d}}|\dot{\Delta}%
_{k}(\nabla \cdot f(u))\dot{\Delta}_{k}u|dx  \notag \\ 
&\leq &\Vert \dot{\Delta}_{k}u\Vert _{L^{2}}\Vert \dot{\Delta}_{k}(\nabla \cdot f(u))\Vert _{L^{2}}.
\end{eqnarray}%
Hence, using Bernstein inequality, \eqref{First_k_estimate_Non} becomes
\begin{equation}  \label{First_k_estimate_Non_2_1}
\frac{1}{2}\ddt\Vert \dot{\Delta}_{k}u\Vert _{L^{2}}^2+2^{2k}\Vert \dot{\Delta}%
_{k}u\Vert _{L^{2}}^2\leq C 2^k\Vert \dot{\Delta}_{k}u\Vert _{L^{2}}\Vert \dot{\Delta}_{k}( f(u))\Vert
_{L^{2}}.  
\end{equation}%
Due to the presence of the prefactor $2^k$ in the second term on the right-hand side of \eqref{First_k_estimate_Non_2_1}, it is suitable to separate the regularity exponents of low and high frequencies. 
 This motivates the introduction of the hybrid Besov spaces in Section \ref{Hybrid_Besov}. 

Using Lemma \ref{Lemma_Diff_Ineq}, and   multiplying  \eqref{First_k_estimate_Non_2_1} by $2^{ks}$ ($k<0$),  
summing in $k$,  we get
\begin{equation}\label{Estimate_main_low}
\Vert u(t)\Vert_{\dot{{B}}_{2,1}^{s,L}}+\int_0^t\Vert u(s)\Vert
_{\dot{{B}}_{2,1}^{s+2,L}}\ds\lesssim \Vert u_{0}\Vert _{%
\dot{B}_{2,1}^{s,L}}+\int_0^t\Vert  f(u)\Vert _{\dot{B}%
_{2,1}^{s+1,L}}\ds.  
\end{equation}%

%
%
%
%
%
%
%
%
%
%
%
%
%
%
%
%
Applying Lemma \ref{Lemma_composition_L}, we obtain
\begin{equation}
\begin{aligned}
&\Vert u(t)\Vert_{\dot{{B}}_{2,1}^{s,L}}+\int_0^t\Vert u(s)\Vert
_{\dot{{B}}_{2,1}^{s+2,L}}\ds\\
\lesssim&\, \Vert u_{0}\Vert _{%
\dot{B}_{2,1}^{s,L}}+\int_0^t (\|u\|_{\dot{B}%
_{2,1}^{\frac{d}{2},L}}+\|u\|_{\dot{B}%
_{2,1}^{\frac{d}{2},H}})\|u\|_{\dot{B}%
_{2,1}^{s+1,L}}\ds\\
&+\int_0^t(\|u\|_{\dot{B}%
_{2,1}^{\frac{d}{2}-1,L}}+\|u\|_{\dot{B}%
_{2,1}^{\frac{d}{2},H}})\|u\|_{\dot{B}%
_{2,1}^{\sigma,H}}\ds \label{Low_Fre_Est_1}
\end{aligned}
\end{equation}%
for some $\sigma\in \R$. 
\smallbreak
 \textbf{High-frequency a priori estimates:}
 In the high-frequency regime, we cannot consider the nonlinearities as source terms, since this would lead to a loss of one derivative. To prevent this,
we rely on commutator estimates.

%
%
We rewrite equation \eqref{Main_N_problem_k} in the form  
\begin{equation}
\begin{array}{ll}
\partial _{t}\dot{\Delta}_{k}u-\Delta P\dot{\Delta}_{k}u+ f'(u)\nabla \dot{\Delta}%
_{k}u=[f'(u), \dot{\Delta}_{k} ]\nabla u. 
\end{array}%
  \label{Main_N_problem_k_comm}
\end{equation}

For $k\geq 0$, multiplying \eqref{Main_N_problem_k_comm} by $\dot{\Delta}_{k}u$ and  taking the $L^2$ inner product, integrating by parts and using Lemma \ref{Bernstein_inequality}, we obtain 
\begin{equation} 
\begin{aligned} \label{First_k_estimate_Non_2_2}
&\frac{1}{2}\ddt\Vert \dot{\Delta}_{k}u\Vert _{L^{2}}^2+\Vert \dot{\Delta}%
_{k}u\Vert _{L^{2}}^2\\
\leq&\, C\Vert \dot{\Delta}_{k}u\Vert _{L^{2}}\left(\|\nabla\cdot f^{\prime}(u)\|_{L^\infty}\| \dot{\Delta}_{k} u\|_{L^2}+\Vert [f'(u), \dot{\Delta}_{k} ]\nabla u\Vert
_{L^{2}}\right)
\end{aligned}
\end{equation}
where we have used the fact that 
\begin{equation}
\int_{\R^d}f'(u)\nabla \dot{\Delta}%
_{k}u  \dot{\Delta}%
_{k}u \dx=-\frac{1}{2}\int_{\R^d} \nabla\cdot f^{\prime}(u) |\dot{\Delta}_{k} u|^2\dx. 
\end{equation}
Using the embedding $\dot{B}_{2,1}^{d/2} (\R^d)\hookrightarrow  L^\infty (\R^d)$ we have  
\begin{equation}\label{Gra_Estima_1}
\|\nabla\cdot f^{\prime}(u)\|_{L^\infty}\| \dot{\Delta}_{k} u\|_{L^2}\lesssim c_k2^{-\sigma k}\|\nabla\cdot f^\prime(u)\|_{\dot{B}_{2,1}^{d/2}}\Vert u\Vert_{\dot{B}_{2,1}^\sigma} 
\end{equation}

To estimate the commutator term in \eqref{First_k_estimate_Non_2_2}, we use Lemma \ref{Commutator_estimate}, to get 
\begin{equation}\label{Grad_2}
\Vert [f'(u), \dot{\Delta}_{k} ]\nabla u\Vert
_{L^{2}}\lesssim c_k2^{-k\sigma}\|\nabla \cdot f'(u)\|_{\dot{B}%
_{2,1}^{d/2}}\|u\|_{\dot{B}%
_{2,1}^{\sigma}}
\end{equation}

Applying Lemma \ref{Lemma_Diff_Ineq} and multiplying \eqref{First_k_estimate_Non_2_2} 
by $2^{k\sigma }$  and summing in $k$, we get 
\begin{equation}\label{Higher_Order_Estimate_M_1}
\begin{aligned}
\Vert u(t)\Vert _{\dot{{B}}_{2,1}^{\sigma ,H}}+\int_0^t\Vert
u\Vert _{\dot{{B}}_{2,1}^{\sigma ,H}}\ds
\lesssim \Vert
u_{0}\Vert _{\dot{{B}}_{2,1}^{\sigma ,H}}
+ \int_0^t \|\nabla\cdot  f'(u)\|_{\dot{B}%
_{2,1}^{d/2}}\|u\|_{\dot{B}%
_{2,1}^{\sigma}}\ds
\end{aligned}
\end{equation}
Applying Lemma \ref{Composition_classical}, we have 
\begin{equation}\label{Grad_Est_3}
\|\nabla\cdot  f'(u)\|_{\dot{B}%
_{2,1}^{d/2}}\lesssim C(\|u\|_{L^\infty}) \|u\|_{\dot{B}%
_{2,1}^{d/2+1}}
\end{equation}
Inserting this into \eqref{Higher_Order_Estimate_M_1}, we obtain
\begin{equation}\label{Higher_Order_Estimate_M_2}
\begin{aligned}
\Vert u(t)\Vert _{\dot{{B}}_{2,1}^{\sigma ,H}}+\int_0^t\Vert
u\Vert _{\dot{{B}}_{2,1}^{\sigma ,H}}\ds
\lesssim \Vert
u_{0}\Vert _{\dot{{B}}_{2,1}^{\sigma ,H}}
+ \int_0^t \|u\|_{\dot{B}%
_{2,1}^{\frac{d}{2}+1}}\|u\|_{\dot{B}%
_{2,1}^{\sigma}}\ds
\end{aligned}
\end{equation}
\textbf{Gathering the estimates:}
Taking $\sigma=\frac{d}{2}+1$ and $s=\frac{d}{2}-1$, from \eqref{Higher_Order_Estimate_M_2} and \eqref{Low_Fre_Est_1}, we obtain 
\begin{equation}\label{Estimate_bootstrap}
\begin{aligned}
\mathcal{E}(t)+\mathcal{D}(t)\lesssim&\, \|u_0\|_{X}+\int_0^t (\|u\|_{\dot{B}%
_{2,1}^{\frac{d}{2},L}}+\|u\|_{\dot{B}%
_{2,1}^{\frac{d}{2}+1,H}})\|u\|_{\dot{B}%
_{2,1}^{\frac{d}{2},L}}\ds\\
&+\int_0^t(\|u\|_{\dot{B}%
_{2,1}^{\frac{d}{2}-1,L}}+\|u\|_{\dot{B}%
_{2,1}^{\frac{d}{2}+1,H}})\|u\|_{\dot{B}%
_{2,1}^{\frac{d}{2}+1,H}}\ds.
\end{aligned}
\end{equation}
We have
\begin{equation}\label{Estimate_integ}
 \begin{aligned}
&\int_0^t \|u\|_{\dot{B}%
_{2,1}^{\frac{d}{2}+1,H}}\|u\|_{\dot{B}%
_{2,1}^{\frac{d}{2}-1,L}}
+(\|u\|_{\dot{B}%
_{2,1}^{\frac{d}{2}-1,L}}+\|u\|_{\dot{B}%
_{2,1}^{\frac{d}{2}+1,H}})\|u\|_{\dot{B}%
_{2,1}^{\frac{d}{2}+1,H}}
\lesssim \mathcal{E}(t)\mathcal{D}(t). 
\end{aligned}  
\end{equation}
Now, we need to estimate the second term on the right-hand side of \eqref{Estimate_bootstrap}. Using
\begin{equation}
\|u\|_{\dot{B}%
_{2,1}^{\frac{d}{2},L}}\lesssim \|u\|_{\dot{B}%
_{2,1}^{\frac{d}{2}}}
\end{equation}
together with the interpolation inequality  
\begin{equation}\label{Interpolation_ineq}
\|u\|_{L_t^2(\dot{B}^{\frac{d}{2}}_{2,1})} 
\lesssim 
\sqrt{\, \|u\|_{L_t^\infty(\dot{B}^{\frac{d}{2}-1}_{2,1})} 
      \, \|u\|_{L_t^1(\dot{B}^{\frac{d}{2}+1}_{2,1})} },
\end{equation}
we obtain 
\begin{equation}\label{Estimate_Term_low}
\begin{aligned}
\int_0^t\|u(s)\|_{B^{\frac{d}{2},L}}^2\ds \lesssim&\, \|u\|_{L_t^2(\dot{B}^{\frac{d}{2}}_{2,1})}^2\\
\lesssim &\,\|u\|_{L_t^\infty(\dot{B}^{\frac{d}{2}-1}_{2,1})} 
      \, \|u\|_{L_t^1(\dot{B}^{\frac{d}{2}+1}_{2,1})}\\
    \lesssim &\, (\|u\|_{L_t^\infty(\dot{B}^{\frac{d}{2}-1,L}_{2,1})}+ \|u\|_{L_t^\infty(\dot{B}^{\frac{d}{2}-1,H}_{2,1})})\|u\|_{L_t^1(\dot{B}^{\frac{d}{2}+1}_{2,1})}\\
 \lesssim &\, (\|u\|_{L_t^\infty(\dot{B}^{\frac{d}{2}-1,L}_{2,1})}+ \|u\|_{L_t^\infty(\dot{B}^{\frac{d}{2}+1,H}_{2,1})})\|u\|_{L_t^1(\dot{B}^{\frac{d}{2}+1}_{2,1})} \\
 \lesssim \,&\mathcal{E}(t)\mathcal{D}(t),   
\end{aligned}
\end{equation}
where we have used \eqref{Inequality_h_low}.

Inserting \eqref{Estimate_integ} and \eqref{Estimate_Term_low} into \eqref{Estimate_bootstrap}, we obtain 
 \begin{equation}\label{Estimate_bootstrap_main}
\begin{aligned}
\mathcal{E}(t)+\mathcal{D}(t)\lesssim&\, \|u_0\|_{X}+\mathcal{E}(t)\mathcal{D}(t). 
\end{aligned}
\end{equation}
From \eqref{Estimate_bootstrap_main}, we conclude in a standard way that there exists $\alpha>0$ small enough such that if $\|u_0\|_{X}\leq \alpha $, then there exists a $K>0$, independent of $T$,  such that
\begin{equation}\label{Main_Estimate}
\mathcal{E}(t)+\mathcal{D}(t)\leq K\|u_0\|_{X},
\end{equation}
for all $t\in [0,T]$. This uniform estimate allows us to continue the local solution to $T=+\infty$.

\end{proof}
\section{Decay estimates}\label{Section_Decay}

In this section, we investigate the large-time behavior of the global solution constructed in Theorem \ref{Main_Theorem_Nonl}. Our main goal is to derive optimal decay rates for the solution in both $L^2$ and suitable Besov norms. We first establish in Section \ref{Section_Decay_rate_1} a decay estimate of the $L^2$-norm of the solution under the additional assumption $u_0\in \dot{{B}}_{2,\infty}^{-d/2}(\mathbb{R}^d) $. This allows us to recover the decay rate $(1+t)^{-d/4}$ which is consistent with the decay rate of the heat kernel.   In Section \ref{Section_Decay_rate_Besov}, we extend this result to higher order norms and derive decay estimates in both low-frequency and high-frequency Besov spaces. These decay estimates reflect the different smoothing effect of the operator $\Delta P$ and provide a refined description of the asymptotic behavior of the solution.    In Section \ref{Optimal_Decay_Section}, we address the optimality of the decay rate and show under the assumption $\int_{\R^d} u_0(x)\dx\neq 0$, using the Fourier splitting method \cite{Schonb_1991}, that the decay rate obtained in Section \ref{Section_Decay_rate_1} is optimal and cannot be improved. 
Finally, in Section \ref{Section_Improved_Decay}, we derive an improved decay estimate  for initial data satisfying $\int_{\R^d}u_0(x)\dx=0$ and $u_0\in L^1(\R^d, (1+|x|)\dx)$. For such data, the decay rate is further enhanced by an additional factor of  $t^{-1/2}$.

\subsection{Decay estimates for $u_0\in L^2(\R^d)\cap \dot{{B}}_{2,\infty}^{-d/2}(\mathbb{R}^d)$}\label{Section_Decay_rate_1}
In this section, we prove the decay estimates of the $L^2$ norm of the solution, for initial data $u_0\in L^2(\R^d)\cap \dot{{B}}_{2,\infty}^{-d/2}(\mathbb{R}^d)$.
 
More precisely, we have the following decay estimate.  
\begin{theorem}\label{theorem_Decay}
Let $u$ be the global solution of \eqref{Main_problem} obtained in Theorem \ref{Main_Theorem_Nonl}. Assume in addition that  $u_0\in \dot{{B}}_{2,\infty}^{-d/2}(\mathbb{R}^d)$. Then, it holds that 
\begin{equation}\label{Estimate_Decay_Main}
\|u(t)\|_{L^2}\lesssim \|u_0\|_{L^2\cap \dot{{B}}_{2,\infty}^{-d/2}(\mathbb{R}^d)}(1+t)^{-\frac{d}{4}}. 
\end{equation}

\end{theorem}\label{Theorem_1}
 The proof of Theorem \ref{theorem_Decay} is based on a method similar to the classical method of Nash \cite{Nash_1}, used to prove the decay of parabolic equations. Similar decay estimates have been obtained in \cite{Duan_Klem_Zhu_2010} under additional assumptions on the initial data, i.e., $u_0\in L^1(\mathbb{R}^d)$. Here, we assume  that the initial data  $u_0\in \dot{{B}}_{2,\infty}^{-d/2}$ which is a weaker assumption than $u_0\in L^1(\R^d)$ due to the embedding $L^1(\R^d)\hookrightarrow\dot{{B}}_{2,\infty}^{-d/2}(\mathbb{R}^d)$.  
 
   Two important ingredients are needed here. First, we need a modified nonlocal Nash inequality that replaces the classical Nash inequality:  
\begin{equation}
\|u\|_{L^2}^{2 }
    \lesssim 
    \|\nabla u\|_{L^2}^{\frac{2d}{d+2}}\|u\|_{L^1}^{\,\frac{4}{d+2}},\quad d\geq 1,\quad u\in H^1(\R^d) 
\end{equation}
and secondly, we need to propagate the $\dot{{B}}_{2,\infty}^{-d/2}$-norm by showing that for all $t\geq 0$, we have  
\begin{equation}
\|u(t)\|_{\dot{{B}}_{2,\infty}^{-d/2}}\lesssim  \|u_0\|_{\dot{{B}}_{2,\infty}^{-d/2}}. 
\end{equation}
These two results are proved  in Lemma \ref{Lemma_Nash} and Lemma \ref{negative_Besov_Norm}. With these two lemmas at hand, we adapt the method in \cite{Duan_Klem_Zhu_2010} to prove  the decay rate of the solution.

\begin{lemma}\label{Lemma_Nash}
It holds that
\begin{equation}
\int_{\R^d}\frac{1}{1+|\xi|^2}|\hat{u}|^2d\xi\leq C \left(\int_{\R^d}\frac{|\xi|^2}{1+|\xi|^2}|\hat{u}|^2 d\xi\right)^{\frac{d}{d+2}}\|u\|_{\dot{{B}}_{2,\infty}^{-d/2}}^{\frac{4}{d+2}}. 
\end{equation} 
\end{lemma}
\begin{proof}
    Let $R>0$. We have 
\begin{equation}
\begin{aligned}
I = \int_{\mathbb{R}^d } \frac{1}{1+|\xi|^2} |\widehat{u}|^2
= &\,\int_{|\xi|\le R} \frac{1}{1+|\xi|^2} |\widehat{u}|^2
+ \int_{|\xi|>R} \frac{1}{1+|\xi|^2} |\widehat{u}|^2\\
=:&\, I_{\mathrm{low}} + I_{\mathrm{high}}.
\end{aligned}
\end{equation}
\medskip
\noindent\textbf{Step 1} (Low-frequency estimates):
Since $\dfrac{1}{1+|\xi|^2} \le 1$, we have
\[
I_{\mathrm{low}}
\le \int_{|\xi|\le R} |\widehat{u}|^2
\lesssim \sum_{2^q \lesssim R} \|\dot\Delta_q u\|_{L^2}^2
\le \left( \sup_q 2^{-dq} \|\dot\Delta_q u\|_{L^2}^2 \right)
\sum_{2^q \lesssim R} 2^{dq}.
\]
By the definition of negative Besov norms, we obtain
\[
I_{\mathrm{low}} \lesssim \dfrac{R^d}{2^d-1} \|u\|_{\dot B_{2,\infty}^{-\frac{d}{2}}}^2.
\]

\noindent\textbf{Step 2} (High-frequency estimates):
For $|\xi| > R$, one has
\[
I_{\mathrm{high}} \le \frac{1}{R^2} \int_{\mathbb{R}^d } \frac{|\xi|^2}{1+|\xi|^2} |\widehat{u}|^2
= \frac{A}{R^2},
\]
where we denote
\[
A := \int_{\mathbb{R}^d } \frac{|\xi|^2}{1+|\xi|^2} |\widehat{u}(\xi)|^2\, d\xi.
\]
\medskip
Combining both parts, we obtain
\[
I \lesssim R^d B + \frac{A}{R^2},
\qquad
B := \|u\|_{\dot B_{2,\infty}^{-\frac{d}{2}}}^2.
\]
The right-hand side is minimized when $R^{d+2} \sim A/B$, which yields
\[
I \lesssim A^{\frac{d}{d+2}} B^{\frac{2}{d+2}}.
\]
This completes the proof.   
\end{proof}

In the following lemma, we show that we can propagate the Besov norm  $\dot{{B}}_{2,\infty}^{-\sigma}$ for $\sigma\in (-\frac{d}{2},\frac{d}{2}]$, 
 which is a crucial step in the proof of the decay estimate (see \cite[Theorem 2.2]{CRINBARAT20221} and \cite{XuXin}).
\begin{lemma}\label{negative_Besov_Norm}
Let $-d/2<\sigma\leq d/2$. 
Let $u$ be a solution of \eqref{Main_problem}, then we have 
\begin{equation}\label{Propagation_regularity}
\|u\|_{\dot{{B}}_{2,\infty}^{-\sigma}}\leq C\|u_0\|_{\dot{{B}}_{2,\infty}^{-\sigma}}.
\end{equation}

\end{lemma}

\begin{proof}
 The proof follows the method developed in \cite{XuXin}. We have
\begin{equation}\label{Main_Inequality_prop}
\|u\|_{\dot{{B}}_{2,\infty}^{-\sigma}}\lesssim \|u_0\|_{\dot{{B}}_{2,\infty}^{-\sigma}}+\int_0^t\sup_{k\in \Z} 2^{-\sigma k}\left(\|\nabla f^{\prime}(u)\|_{L^\infty}\| \dot{\Delta}_{k} u\|_{L^2}+\Vert [f'(u), \dot{\Delta}_{k} ]\nabla u\Vert
_{L^{2}}\right)\dt. 
\end{equation}
Applying  \eqref{Commu_s}, we obtain 
\begin{equation}
\sup_{k\in \Z} 2^{-\sigma k}\Vert [f'(u), \dot{\Delta}_{k} ]\nabla u\Vert
_{L^{2}}\lesssim C \|\nabla f'(u)\|_{\dot{B}%
_{2,1}^{\frac{d}{2}}}\|u\|_{\dot{B}%
_{2,\infty}^{-\sigma}}. 
\end{equation}
 Since $f'$ is smooth with $f'(0)=0$, we have
 \begin{equation}
\|\nabla f'(u)\|_{\dot{B}%
_{2,1}^{\frac{d}{2}}}\leq \| f'(u)\|_{\dot{B}%
_{2,1}^{\frac{d}{2}+1}}\leq  C(\|u\|_{L^\infty})\|u\|_{\dot{B}%
_{2,1}^{\frac{d}{2}+1}}.
\end{equation}
We have
\begin{equation}\label{Main_Inequality_prop_1}
\|u\|_{\dot{{B}}_{2,\infty}^{-\sigma}}\lesssim \|u_0\|_{\dot{{B}}_{2,\infty}^{-\sigma}}+\int_0^t\|u\|_{\dot{B}%
_{2,1}^{\frac{d}{2}+1}} \|u\|_{\dot{B}%
_{2,\infty}^{-\sigma}}\dt. 
\end{equation}
Applying Gronwall's inequality, we obtain 
\begin{equation}\label{Ineq_unif}
\|u\|_{\dot{{B}}_{2,\infty}^{-\sigma}}\lesssim \|u_0\|_{\dot{{B}}_{2,\infty}^{-\sigma}}\exp\left(\int_0^t\|u\|_{\dot{B}%
_{2,1}^{\frac{d}{2}+1}}\ds\right).
\end{equation}
By Theorem \ref{Main_Theorem_Nonl}, we have 
$u\in L^{1}((0,\infty);\dot{B}_{2,1}^{\frac{d}{2}+1}(%
\mathbb{R}^{d}))$. Hence, \eqref{Ineq_unif} implies that $\|u\|_{\dot{{B}}_{2,\infty}^{-\sigma}}$ is uniformly bounded and then \eqref{Propagation_regularity} is satisfied. 
\end{proof}

\begin{proof}[Proof of Theorem \ref{theorem_Decay}]
Having outlined the general idea above, we proceed with the proof of \eqref{Estimate_Decay_Main}. 
First, it holds that  for all $t\geq 0$ (see \cite[Lemma 2.1]{Duan_Klem_Zhu_2010})
\begin{equation}\label{a_priori_bound}
    \frac{1}{2}\ddt\|u(t)\|_{L^2}^2+\int_{\R^d}\frac{|\xi|^2}{1+|\xi|^2}|\hat{u}(\xi)|^2 d\xi=0.
\end{equation}
Multiplying \eqref{a_priori_bound} by $(1+t)^{\alpha}$ and integrating with respect to time, we obtain 
\begin{equation}\label{time_Weighted_Est_1}
\begin{aligned}
&(1+t)^{\alpha}\|u(t)\|_{L^2}^2+2\int_0^t\int_{\R^d}(1+s)^\alpha \frac{|\xi|^2}{1+|\xi|^2}|\hat{u}(\xi)|^2 d\xi d s\\
\leq&\, \|u_0\|_{L^2}^2+\alpha \int_0^t (1+s)^{\alpha-1}\|u(s)\|_{L^2}^2 ds.
\end{aligned}
\end{equation}
To estimate the last term on the right-hand side of \eqref{time_Weighted_Est_1}, we split it as 
\begin{equation}
\begin{aligned}
   \int_0^t (1+s)^{\alpha-1}\|u(s)\|_{L^2}^2 ds&\,=\int_0^t\int_{\R^d} (1+s)^{\alpha-1}\frac{|\xi|^2}{1+|\xi|^2}|\widehat u(\xi,s)|^2 ds\\
   &+\int_0^t \int_{\mathbb{R}^d}(1+s)^{\alpha-1}\frac{1}{1+|\xi|^2}|\widehat u(\xi,s)|^2 ds\\
   &\, := \rm{J}_1+\rm{J}_2.
    \end{aligned}
\end{equation}
The term $\rm{J}_1$ is estimated using the dissipation in \eqref{time_Weighted_Est_1}, whereas $\rm{J}_2$ is estimated using Lemmas \ref{Lemma_Nash}  and \ref{negative_Besov_Norm}. As in \cite{Duan_Klem_Zhu_2010}, we have for any $\varepsilon>0$
\begin{equation}\label{J_1_Estimate}
    {\rm{J}}_1\leq  \varepsilon \int_0^t\int_{\R^d}(1+s)^\alpha\frac{|\xi|^2}{1+|\xi|^2}|\hat{u}(\xi)|^2 d\xi d s+C(\varepsilon) \|u_0\|_{L^2}^2.
\end{equation}
To show \eqref{J_1_Estimate}, we have for $0\leq \alpha\leq 1$ 
\begin{equation}
    {\rm{J}_1}\leq \int_0^t\int_{\R^d}(1+s)^\alpha\frac{|\xi|^2}{1+|\xi|^2}|\hat{u}(\xi)|^2 d\xi d s\leq \frac{1}{2} \|u_0\|_{L^2}^2. 
\end{equation}
where we have used  \eqref{time_Weighted_Est_1}(with $\alpha=0$). For $\alpha>1$, we have by using Young's inequality with the exponents $\frac{\alpha-1}{\alpha}+\frac{1}{\alpha}=1$:  
\begin{equation}
\begin{aligned}
    {\rm{J}_1}\leq \varepsilon \int_0^t\int_{\R^d}(1+s)^\alpha\frac{|\xi|^2}{1+|\xi|^2}|\hat{u}(\xi)|^2 d\xi d s+C(\varepsilon) \int_0^t\int_{\R^d}\frac{|\xi|^2}{1+|\xi|^2}|\hat{u} (\xi)|^2 d\xi d s\\ 
    \leq  \varepsilon \int_0^t\int_{\R^d}(1+s)^\alpha\frac{|\xi|^2}{1+|\xi|^2}|\hat{u}(\xi)|^2 d\xi d s+C(\varepsilon) \|u_0\|_{L^2}^2. 
    \end{aligned}
\end{equation}
Now, for $\rm{J}_2$, we use Lemma \ref{Lemma_Nash}, to get  
\begin{equation}
{\rm{J}_2}\lesssim \int_0^t (1+s)^{\alpha-1}\left(\int_{\R^d}\frac{|\xi|^2}{1+|\xi|^2}|\hat{u}|^2 d\xi\right)^{\frac{d}{d+2}}\|u\|_{\dot{{B}}_{2,\infty}^{-d/2}}^{\frac{4}{d+2}}ds
\end{equation}
Using Young's inequality with the exponents $\frac{d}{d+2}+\frac{2}{d+2}=1$ together with Lemma \ref{negative_Besov_Norm}, we obtain 
\begin{equation}\label{J_2_Estimate}
\begin{aligned}
{\rm{J}_2}\leq&\, \varepsilon \int_0^t \int_{\R^d}(1+s)^{\alpha}\frac{|\xi|^2}{1+|\xi|^2}|\hat{u}|^2 d\xi ds+C(\varepsilon)\|u_0\|_{\dot{{B}}_{2,\infty}^{-d/2}}^2\int_0^t (1+s)^{\alpha-\frac{d}{2}-1}ds\\
\leq&\, \varepsilon \int_0^t \int_{\R^d}(1+s)^{\alpha}\frac{|\xi|^2}{1+|\xi|^2}|\hat{u}|^2 d\xi ds+C(\varepsilon)\|u_0\|_{\dot{{B}}_{2,\infty}^{-d/2}}^2(1+t)^{\alpha-\frac{d}{2}}, 
\end{aligned}
\end{equation}
provided that $\alpha>\frac{d}{2}$. Hence, selecting $\varepsilon$ small enough and putting together the estimates \eqref{time_Weighted_Est_1}, \eqref{J_1_Estimate} and \eqref{J_2_Estimate}, we obtain  
\begin{equation}\label{time_Weighted_Est_2}
\begin{aligned}
&(1+t)^{\alpha}\|u(t)\|_{L^2}^2+2\int_0^t\int_{\R^d}(1+s)^\alpha \frac{|\xi|^2}{1+|\xi|^2}|\hat{u}(\xi)|^2 d\xi d s\lesssim \|u_0\|_{L^2}^2+\|u_0\|_{\dot{{B}}_{2,\infty}^{-d/2}}^2(1+t)^{\alpha-\frac{d}{2}}.
\end{aligned}
\end{equation}
with $\alpha>\frac{d}{2}$ and any $t\geq 0$. This yields \eqref{Estimate_Decay_Main} and finishes the proof of Theorem \ref{theorem_Decay}. 
\end{proof}

\subsection{Decay estimate of higher-order norm}\label{Section_Decay_rate_Besov}
In this section, and by adapting the method in \cite{Danchin_EMS}, we give a decay rate for the higher-order norms of the solution in the hybrid Besov space $\dot{B}_{2,1}^{d/2-1,d/2+1}$. More precisely, we have the following decay estimate.
\begin{theorem}\label{Theorem_Decay_2}
Suppose that the assumptions of Theorem \ref{Main_Theorem_Nonl} hold. Let $\sigma_1\in (1-\frac{d}{2}, \frac{d}{2}]$. Assume in addition that $u_0\in \dot B^{-\sigma_1}_{2,\infty}$. 
Let 
\begin{equation}
\theta_0=\frac{4}{d+2 \sigma_1 +2}\qquad \alpha_1=\frac{1}{2}(\frac{d}{2}+\sigma_1-1)\quad \text{and}\quad c_0=\left(\|u_0\|_{\dot B^{-\sigma_1}_{2,\infty}}+\Vert u_0\Vert _{X}\right)^{-\frac{\theta_0}{1-\theta_0}}.
\end{equation}
Then, it holds that 
\begin{subequations}
\begin{equation} \label{Decay_Rate_Besov}
\|u(t)\|_{\mathcal{X}}\lesssim (1+c_0 t)^{-\alpha_1}
\end{equation}
and 
\begin{equation}\label{Decay_high_Freq}
\Vert u(t)\Vert _{\dot{{B}}_{2,1}^{d/2+1 ,H}}
\lesssim\,(1+c_0t)^{-2\alpha_1}.
\end{equation}
\end{subequations}
In addition, for $-\sigma_1<\sigma\leq \frac{d}{2}-1$, we have
\begin{subequations}
\begin{equation}\label{Decay_Sigma}
\Vert u(t)\Vert_{\dot B^{\sigma,L}_{2,1}}\leq C(\Vert u_0\Vert_{\dot B^{-\sigma_1}_{2,\infty}}, \|u_0\|_X)(1+ t)^{-\frac{\sigma_1+\sigma}{2}}. 
\end{equation}
 Also, if $u_0\in L^2\cap\dot B^{-d/2}_{2,\infty} $, then for  any $\sigma'\in (0,d/2)$, it holds that 
 \begin{equation}\label{Decay_Negative_Besov}
\|u(t)\|_{\dot B^{-\sigma'}_{2,\infty} }\leq C(\|u_0\|_{\dot B^{-d/2}_{2,\infty} }, \|u_0\|_{L^2}) (1+t)^{-(\frac{d}{4}-\frac{\sigma'}{2})}.
\end{equation}
 \end{subequations}

\end{theorem}
\begin{remark}
The decay rate in \eqref{Decay_Sigma} is the same decay rate obtained in \eqref{Main_Decay_estimate}  for the low frequency part of  the linear problem and is consistent with the optimal decay rate for the heat equation \cite{Zha_JMP,Guo_Wang_2012}.  Although the decay rate of the high frequency part for the linearized problem is exponentially fast, this is not the case for the nonlinear problem due to nonlinear interaction between the low and high frequency terms \cite{Danchin_EMS}. 
\end{remark}
\begin{remark}
Using the interpolation inequality \begin{equation}
\|u\|_{\dot B^{0}_{2,2}}
\;\le\;
\|u\|_{\dot B^{-\sigma'}_{2,\infty}}^{1/2}
\,
\|u\|_{\dot B^{\frac d2-1}_{2,1}}^{1/2}\qquad\text{with}\quad  
\sigma'=\frac{d}{2}-1
\end{equation}
together with the fact that $L^2\approx B_{2,2}^0$ (see \cite[Proposition 2.39]{Bahouri_2011_1}),  the estimates \eqref{Decay_Rate_Besov} and \eqref{Decay_Negative_Besov}, we obtain, for $\sigma_1=d/2$,
\begin{equation}
\|u(t)\|_{L^2}\lesssim (1+c_0 t)^{-\alpha_1(1-\theta)}(1+t)^{-\theta(\frac{d}{4}-\frac{\sigma'}{2})}\leq C(c_0, \|u_0\|_{L^2}) (1+c_0t)^{-d/4},\quad d>2.
\end{equation}
which is the same decay rate obtained in \eqref{Estimate_Decay_Main} although \eqref{Estimate_Decay_Main} also holds for $d=2$. 
\end{remark}
\begin{remark}
Using the interpolation inequality \begin{equation}
\|u\|_{L^p}
\le
C\,\|u\|_{L^2}^{\frac{2(p-1)}{p}}
\|u\|_{L^1}^{\frac{2-p}{p}},\quad 1<p<2
\end{equation}
together with \eqref{Contra_1} and the decay estimate in \cite[Theorem 2.1]{Duan_Klem_Zhu_2010}, we obtain the $L^1-L^p$ decay estimate 
\begin{equation}\label{L_p_Decay}
\|u\|_{L^p}\leq C(\|u_0\|_{L^1}, \|u_0\|_{L^2}) (1+t)^{-\frac{d}{2}(1-\frac{1}{p})}.
\end{equation}
Now, using \eqref{Decay_Negative_Besov} for $\sigma'=\frac{d}{p}-\frac{d}{2}$, we obtain for $p\in (1,2)$
 \begin{equation}\label{Decay_L_p}
\|u(t)\|_{\dot B^{\frac{d}{2}-\frac{d}{p}}_{2,\infty} }\leq C(\|u_0\|_{\dot B^{-d/2}_{2,\infty} }, \|u_0\|_{L^2}) (1+t)^{-\frac{d}{2}(1-\frac{1}{p})},\quad 
\end{equation} 
which is the same 
   optimal $L^1-L^p$ decay rate in \eqref{L_p_Decay}  in the larger spaces $\dot B^{\frac{d}{p}-\frac{d}{2}}_{2,\infty}(\R^d)\supset L^p(\R^d) $ under  a weaker assumption on the initial data. i.e., $u_0\in \dot B^{-\frac{d}{2}}_{2,\infty}(\R^d)\supset L^1(\R^d)$.  
\end{remark}
\begin{proof}[Proof of Theorem \ref{Theorem_Decay_2}]
We prove the decay estimate using an interpolation argument as in \cite{XuXin} and \cite{Danchin_EMS}. This method allows $\|u_0\|_{\dot B^{-\sigma_1}_{2,\infty}}$ to be large, hence, only $\|u_0\|_X $ is required to be small. 
By adapting the method in \cite{Danchin_EMS}, and following the same steps as in the proof of the estimate \eqref{Main_Estimate}, we obtain for all $t\geq t_0 \geq 0$   
\begin{equation}\label{Ineq_L_H}
\mathcal{L}(t)+c\int_{t_0}^t \mathcal{H}(t)\dt\lesssim \mathcal{L}(t_0)
\end{equation}
 with 
 \begin{equation}
\mathcal{L}(t)\sim\Vert u(t)\Vert _{%
\dot{B}_{2,1}^{\frac{d}{2}-1,L}}+\Vert
u(t)\Vert _{\dot{{B}}_{2,1}^{\frac{d}{2}+1 ,H}}\quad \text{and}\quad \mathcal{H}(t)=\|u(t)\|_{\dot{{B}}_{2,1}^{\frac{d}{2}+1}}
\end{equation}
Inequality \eqref{Ineq_L_H} implies that $\mathcal{L}$ is nonincreasing on $\R^+$ and hence it is almost everywhere differentiable on $\R^+$. Hence, following an argument from \cite[p.10]{Crin-Bara_Danchin_2021},  \eqref{Ineq_L_H} implies 
\begin{equation}\label{Diff_Ineq}
\ddt \mathcal{L}(t)+c\mathcal{H}(t)\leq 0. 
\end{equation}
Using the interpolation inequalities (\cite{Danchin_EMS,XuXin})
\begin{equation}
\|u\|_{\dot B^{\frac d2-1,L}_{2,1}}
\;\lesssim\;
\left(\|u\|_{\dot B^{\frac d2+1,L}_{2,1}}\right)^{1-\theta_0}
\left(\|u\|_{\dot B^{-\sigma_1,L}_{2,\infty}}\right)^{\theta_0},
\end{equation}
with 
\begin{equation}
-\sigma_1<\frac{d}{2}-1,\qquad \text{and}\qquad (1-\theta_0)\left(1+\frac d2\right)-\sigma_1\theta_0
=\frac d2-1,
\end{equation}
we obtain  by using \eqref{Propagation_regularity}
\begin{equation}\label{L_Interpolation_Est}
\|u(t)\|_{\dot B^{\frac d2+1,L}_{2,1}}
\;\gtrsim\;
\left(\|u(t)\|_{\dot B^{\frac d2-1,L}_{2,1}}\right)^{\frac{1}{1-\theta_0}}
\|u_0\|_{\dot B^{-\sigma_1}_{2,\infty}}^{-\frac{\theta_0}{1-\theta_0}}.
\end{equation}
For the high-frequency part, we have, by using \eqref{Main_Estimate} 
\begin{equation}\label{H_Estimate_interp}
\begin{aligned}
\Vert u\Vert _{\dot{{B}}_{2,1}^{\frac{d}{2}+1 ,H}}=&\, \Vert u\Vert _{\dot{{B}}_{2,1}^{\frac{d}{2}+1 ,H}}^{\frac{1}{1-\theta_0}}\Vert u\Vert _{\dot{{B}}_{2,1}^{\frac{d}{2}+1 ,H}}^{\frac{-\theta_0}{1-\theta_0}}\\
\gtrsim&\,\Vert u\Vert _{\dot{{B}}_{2,1}^{\frac{d}{2}+1 ,H}}^{\frac{1}{1-\theta_0}}\Vert u_0\Vert _{X}^{\frac{-\theta_0}{1-\theta_0}}.
\end{aligned}
\end{equation}
Collecting \eqref{L_Interpolation_Est} and \eqref{H_Estimate_interp}, and using the inequality 
\begin{equation}
(a+b)^p \le (1+\varepsilon)^{p-1} a^p + \left(1+\frac{1}{\varepsilon}\right)^{p-1} b^p,
\quad  1 \le p < \infty,
\end{equation}
which holds for arbitrary $a,b \in \mathbb{R}^+$ and $\varepsilon>0$, we obtain 
\begin{equation}\label{H_L_Est}
\mathcal{H}(t)\geq c \left(\|u_0\|_{\dot B^{-\sigma_1}_{2,\infty}}+\Vert u_0\Vert _{X}\right)^{-\frac{\theta_0}{1-\theta_0}} (\mathcal{L}(t))^{\frac{1}{1-\theta_0}}.
\end{equation}
Combining \eqref{Diff_Ineq} and \eqref{H_L_Est}, we obtain 
\begin{equation}\label{Diff_Ineq_2}
\ddt \mathcal{L}(t)+c_0(\mathcal{L}(t))^{\frac{1}{1-\theta_0}}\leq 0 
\end{equation}
with 
\begin{equation}
c_0=c\left(\|u_0\|_{\dot B^{-\sigma_1}_{2,\infty}}+\Vert u_0\Vert _{X}\right)^{-\frac{\theta_0}{1-\theta_0}}. 
\end{equation}  
Inequality \eqref{Diff_Ineq_2} implies 
\begin{equation}
\mathcal{L}(t)\lesssim \mathcal{L}(0)(1+\dfrac{\theta_0}{1-\theta_0} \mathcal{L}(0)^{\frac{\theta_0}{1-\theta_0}}c_0 t)^{1-\frac{1}{\theta_0}}.
\end{equation}
Using the fact that $\theta_0=\frac{4}{d+2 \sigma_1 +2}\in(0,1)$, we obtain $1-\frac{1}{\theta_0}=\frac{1}{2} (-\frac{d}{2}- \sigma_1 +1)$ and hence,  since $\mathcal{L}(t)\sim \|u(t)\|_{\mathcal{X}}$ we obtain \eqref{Decay_Rate_Besov}. 

 Our goal now is to prove \eqref{Decay_high_Freq}. 
Following the steps leading to \eqref{Higher_Order_Estimate_M_1} we obtain from 
\eqref{First_k_estimate_Non_2_2}, \eqref{Gra_Estima_1}, \eqref{Grad_2} and \eqref{Grad_Est_3} 
\begin{equation}\label{Higher_Order_Estimate_Decay}
\begin{aligned}
\ddt\Vert u(t)\Vert _{\dot{{B}}_{2,1}^{d/2+1 ,H}}+c\Vert
u\Vert _{\dot{{B}}_{2,1}^{d/2+1 ,H}}
\lesssim  \|u\|_{\dot{B}%
_{2,1}^{d/2+1}}^2.  
\end{aligned}
\end{equation}
This yields 
\begin{equation}\label{H_norm_Decay_1}
\begin{aligned}
\Vert u(t)\Vert _{\dot{{B}}_{2,1}^{d/2+1 ,H}}
\lesssim e^{-ct}\Vert u_0\Vert _{\dot{{B}}_{2,1}^{d/2+1 ,H}}+\int_0^t e^{-c(t-s)} \|u(s)\|_{\dot{B}%
_{2,1}^{d/2+1}}^2\ds.
\end{aligned}
\end{equation}
On the other hand we have the estimate  
\begin{equation}\label{Estim_X_norm}
\|u\| _{\dot{{B}}_{2,1}^{d/2+1 }}\lesssim (\|u\| _{\dot{{B}}_{2,1}^{d/2-1 ,L}}+\|u\| _{\dot{{B}}_{2,1}^{d/2+1 ,H}})\lesssim \|u\|_{\mathcal{X}}.
\end{equation}
Let $\alpha_1=\frac{1}{2}(\frac{d}{2}+\sigma_1-1)$.  Multiplying \eqref{H_norm_Decay_1}  by $(1+c_0t)^{2\alpha_1}$ and using  \eqref{Estim_X_norm},  we get 
\begin{equation}\label{H_norm_Decay_2}
\begin{aligned}
(1+c_0t)^{2\alpha_1}\Vert u(t)\Vert _{\dot{{B}}_{2,1}^{d/2+1 ,H}}
\lesssim&\, (1+c_0t)^{2\alpha_1} e^{-ct}\Vert u_0\Vert _{\dot{{B}}_{2,1}^{d/2+1 ,H}}\\
&+\int_0^t \left(\frac{1+c_0t}{1+c_0s }\right)^{2\alpha_1} e^{-c(t-s)} ((1+c_0s)^{2\alpha_1}\|u\|_{\mathcal{X}}^2)\ds. 
\end{aligned} 
\end{equation}
Recalling \eqref{Decay_Rate_Besov}, we have (since $\alpha$ in \eqref{alpha_condition} is small)
\begin{equation}
\sup_{t\geq 0} ((1+c_0t)^{2\alpha_1}\|u\|_{\mathcal{X}}^2)\lesssim \|u_0\|_{X}^2\lesssim \|u_0\|_{X}.
\end{equation}
On the other hand, we have 
\begin{equation}
\int_0^t \left(\frac{1+c_0t}{1+c_0s }\right)^{2\alpha_1} e^{-c(t-s)}ds\lesssim \int_0^t (1+c_0(t-s))^{2\alpha_1}e^{-c(t-s)}\leq C_0
\end{equation}
where $C_0$ depends only on $\alpha_1$ and $c_0$. Hence, putting the above two estimates together, we have from \eqref{H_norm_Decay_2}  
\begin{equation}
\Vert u(t)\Vert _{\dot{{B}}_{2,1}^{d/2+1 ,H}}
\lesssim\,\|u_0\|_{X}(1+c_0t)^{-2\alpha_1}.
\end{equation}
To prove \eqref{Decay_Sigma}, we use interpolation inequality 
\begin{equation}
\Vert u\Vert_{\dot B^{\sigma,L}_{2,1}}\lesssim \Vert u\Vert ^{\nu}_{%
\dot{B}_{2,1}^{\frac{d}{2}-1,L}}\Vert u\Vert_{\dot B^{-\sigma_1}_{2,\infty}}^{1-\nu},\quad \text{with}\quad \nu=\frac{2 \left(\sigma +\sigma _1\right)}{d+2 \sigma _1-2}\in (0,1), 
\end{equation}
together with \eqref{Decay_Rate_Besov} and \eqref{Propagation_regularity} to get 
 \begin{equation}
\Vert u\Vert_{\dot B^{\sigma,L}_{2,1}}\lesssim \Vert u\Vert ^{\nu}_{%
\dot{B}_{2,1}^{\frac{d}{2}-1,L}}\Vert u_0\Vert_{\dot B^{-\sigma_1}_{2,\infty}}^{1-\nu}\lesssim \|u_0\|_X\Vert u_0\Vert_{\dot B^{-\sigma_1}_{2,\infty}}^{1-\nu}(1+c_0 t)^{-\frac{\sigma_1+\sigma}{2}}.  
\end{equation}
This ends the proof of \eqref{Decay_Sigma}. Next, using the interpolation inequality 
\begin{equation}
\|u\|_{\dot B^{-\sigma'}_{2,\infty} }
\;\le\;
\|u\|_{\dot B^{-d/2}_{2,\infty} }^{2\sigma'/d}
\,
\|u\|_{L^2 }^{(d-2\sigma')/d}
\end{equation}
which holds for $\sigma'\in (0,d/2)$. Hence, from \eqref{Propagation_regularity} and \eqref{Estimate_Decay_Main}, we obtain \eqref{Decay_Negative_Besov}.

\end{proof}


\subsection{Optimality of the decay rate for the linear problem}\label{Optimal_Decay_Section}
In this section, we show that if $\int_{\R^d}u_0(x)\dx\neq 0$, then the decay rate in \eqref{Estimate_Decay_Main} is optimal. 
\begin{lemma}\label{Lemma_Opt_Linear}
Assume that there exist $\Lambda_2,\eta>0$ such that $|\hat{u}%
_0(\xi)|\geq \eta$ for $|\xi|\leq \Lambda_2$, then the solution of the linear problem \eqref{Main_problem_Linear} has the following lower bound
\begin{eqnarray}
\left\Vert u(t)\right\Vert _{L^{2}}&\geq
&C_\eta(1+t)^{-\frac{d}{4}},  \label{regularity_gain_estimate_lowe}
\end{eqnarray}%
for any $t\geq 0$ and where $C_\eta$ is a constant
depending on $\eta$ and the fixed parameters.
\end{lemma}
\begin{proof}
The proof follows the ideas from \cite{Schonb_1991}. We have from \eqref{Main_problem_Linear} by Plancherel identity 
\begin{eqnarray}  \label{Iden_Main}
\left\Vert u(t)\right\Vert
_{L^{2}}&=&\Vert e^{-\frac{\left\vert
\xi \right\vert ^{2}}{1+|\xi|^2}t}\hat{u}_{0}\left( \xi \right) \Vert
_{L^{2}}  \notag \\
&\geq&\left(\int_{\left\vert \xi \right\vert \leq \Lambda_2} e^{-2\frac{\left\vert \xi \right\vert ^{2}}{1+|\xi|^2}t}|%
\hat{u}_{0}|^2\left( \xi \right)d\xi\right)^{1/2}  \notag \\
&\geq&\eta\left(\int_{\left\vert \xi \right\vert \leq \Lambda_2}e^{-2|\xi|^2t}d\xi\right)^{1/2},
\end{eqnarray}
where we have used the fact that $|\xi|^2/(1+|\xi|^2)\leq |\xi|^2$. We
rewrite the integral in the last formula as: 
\begin{equation}
\begin{aligned}
\int_{\left\vert \xi \right\vert \leq \Lambda_2}e^{-2|\xi|^2t}d\xi=&\,\int_{0}^{\Lambda_2} r
^{d-1}e^{-2r^2t}dr =C_1 t^{-d/2}\int_{0}^{\Lambda_2\sqrt{t}} y
^{d-1}e^{-2y^2}dy\\
\geq&\, C_1 t^{-d/2}\int_{0}^{1} y
^{d-1}e^{-2y^2}dy,\qquad \text{for}\qquad t\geq \Lambda_2^{-2}\\
\gtrsim&\,  (1+t)^{-d/2}.
\end{aligned}
\end{equation}
For $t<\Lambda_2^{-2}$, we have 
\begin{equation}
\int_{0}^{\Lambda_2} r
^{d-1}e^{-2r^2t}dr\geq (1+t)^{-d/2}\int_{0}^{\Lambda_2} 
r ^{d-1}e^{-2r^2\Lambda_2^{-2}}dr.
\end{equation}
Set
\begin{equation}
C_0=\min\left\{C_1, \int_{0}^{\Lambda_2} r
^{d-1}e^{-2r^2\Lambda_2^{-2}}dr\right\},
\end{equation}
We obtain that 
\begin{eqnarray}
\left\Vert u(t)\right\Vert
_{L^{2}}\gtrsim \left(\int_{\left\vert \xi \right\vert \leq \Lambda_2}e^{-2|\xi|^2t}d\xi\right)^{1/2}\gtrsim(1+t)^{-d/4}.
\end{eqnarray}
From this inequality and \eqref{Iden_Main}, we deduce %
\eqref{regularity_gain_estimate_lowe}.\end{proof}

\subsection{Optimality of the decay rate for the nonlinear problem}
\label{Optimal_Decay_Section_NL}
We set
\[
m(\xi):=\frac{|\xi|^2}{1+|\xi|^2}.
\]

\begin{lemma}[Optimality of the nonlinear decay rate]
\label{Nonlinear_Optimal_Decay_Lemma}
Let \(d\geq 2\), $f(u)=bu^2$ with $b\in \R^d\setminus\{0\}$ and assume that the hypothesis of Theorem \ref{Main_Theorem_Nonl} and Theorem \ref{theorem_Decay} hold. Let \(u\) be a global solution of the nonlinear problem satisfying
\[
\|u(t)\|_{L^2}\lesssim (1+t)^{-\frac d4}.
\]
Assume that there exist \(\eta,\Lambda_2>0\) such that
\[
|\widehat u_0(\xi)|\geq \eta,
\qquad \text{ for}\quad  |\xi|\leq \Lambda_2.
\]
Then there exist \(c_\eta>0\) and \(T_\eta>0\) such that
\[
\|u(t)\|_{L^2}\geq c_\eta (1+t)^{-\frac d4},
\qquad t\geq T_\eta.
\]
Consequently, the decay rate \((1+t)^{-d/4}\) is optimal.
\end{lemma}

\begin{proof}
Let \(v\) be the solution of the linear problem
\[
\left\{
\begin{array}{ll}
v_t-\Delta P v=0,\\
v(x,0)=u_0(x).
\end{array}
\right.
\]
By the linear lower bound proved in the previous section (Lemma \ref{Lemma_Opt_Linear}), the assumption
\[
|\widehat u_0(\xi)|\geq \eta,
\qquad |\xi|\leq \Lambda_2,
\]
implies that
\begin{equation}
\label{Lower_Bound_v_NL}
\|v(t)\|_{L^2}\geq c_\eta (1+t)^{-\frac d4}.
\end{equation}
We now set $w:=u-v.$ Then \(w\) satisfies
\begin{equation}
\label{w_Equation}
\left\{
\begin{array}{ll}
w_t-\Delta P w=-\nabla\cdot f(u),\\
w(x,0)=0.
\end{array}
\right.
\end{equation}
We shall prove that \(w\) decays faster than \(v\). Multiplying
\eqref{w_Equation} by \(w\), integrating over \(\mathbb R^d\), and using
Plancherel's identity, we obtain
\begin{equation}
\label{Energy_w}
\begin{aligned}
\ddt\|w(t)\|_{L^2}^2
+
2\int_{\mathbb R^d}
m(\xi)|\widehat w(\xi)|^2\,d\xi
&=
-2\int_{\mathbb R^d} w\,\nabla\cdot f(u)\,dx \\
&=
-2\int_{\mathbb R^d} (u-v)\,\nabla\cdot f(u)\,dx .
\end{aligned}
\end{equation}
Using the cancellation
\[
\int_{\mathbb R^d} u\,\nabla\cdot f(u)\,dx=0,
\]
we get
\[
-2\int_{\mathbb R^d} (u-v)\,\nabla\cdot f(u)\,dx
=
2\int_{\mathbb R^d} v\,\nabla\cdot f(u)\,dx
=
-2\int_{\mathbb R^d} \nabla v\cdot f(u)\,dx.
\]
Therefore, since $f(u)=bu^2$,
\begin{equation}
\label{Energy_w_1}
\ddt\|w(t)\|_{L^2}^2
+
2\int_{\mathbb R^d}
m(\xi)|\widehat w(\xi)|^2\,d\xi
\lesssim
\|\nabla v(t)\|_{L^\infty}\|u(t)\|_{L^2}^2.
\end{equation}

We now rely on the Fourier splitting method developed in \cite{Schonb_1991}. Define
\[
F(t):=(1+t)^d\|w(t)\|_{L^2}^2.
\]
Multiplying \eqref{Energy_w_1} by \((1+t)^d\), we obtain
\begin{equation}
\label{Energy_w_2}
\begin{aligned}
F'(t)
\leq{}&
d(1+t)^{d-1}\|w(t)\|_{L^2}^2
-
2(1+t)^d
\int_{\mathbb R^d}
m(\xi)|\widehat w(\xi)|^2\,d\xi \\
&+
C(1+t)^d
\|\nabla v(t)\|_{L^\infty}
\|u(t)\|_{L^2}^2 .
\end{aligned}
\end{equation}
For \(t\geq T_0\), with \(T_0\) large enough, define
\begin{equation}
\label{definition_S_t}
S(t):=
\left\{
\xi\in\mathbb R^d:
m(\xi)\leq \frac{d}{2(1+t)}
\right\}.
\end{equation}
Using the fact that 
$
m(\xi)\leq \frac{d}{2(1+t)}
$
implies for $t\geq T_0$ that 
$$|\xi|^2\leq \frac{d}{2(1+t)-d}\lesssim (1+t)^{-1}.$$
Then \(S(t)\) is contained in a ball of radius \(C(1+t)^{-1/2}\). Hence
\begin{equation}
\label{volume_S_t}
|S(t)|\lesssim (1+t)^{-\frac d2},
\qquad
|\xi|\lesssim (1+t)^{-1/2}
\quad\text{for }\xi\in S(t).
\end{equation}
Moreover, on \(S(t)^c\), we have
\[
2(1+t)^d m(\xi)
\geq
d(1+t)^{d-1}.
\]
Therefore, from \eqref{Energy_w_2}, for \(t\geq T_0\),
\begin{equation}
\label{Energy_w_3}
F'(t)
\lesssim
(1+t)^{d-1}
\int_{S(t)}|\widehat w(\xi,t)|^2\,d\xi
+
(1+t)^d
\|\nabla v(t)\|_{L^\infty}
\|u(t)\|_{L^2}^2 .
\end{equation}
Using the linear decay estimate and the embedding
\[
\dot B^{\frac d2}_{2,1}\hookrightarrow L^\infty,
\]
we have
\begin{equation}
\label{Decay_v}
\|\nabla v(t)\|_{L^\infty}
\lesssim
\|v(t)\|_{\dot B^{\frac d2+1}_{2,1}}
\lesssim
(1+t)^{-\frac d2-\frac12}
\left(
\|u_0\|_{\dot B^{-d/2}_{2,\infty}}
+
\|u_0\|_{\dot B^{d/2+1}_{2,1}}
\right).
\end{equation}
Together with
\[
\|u(t)\|_{L^2}^2\lesssim (1+t)^{-\frac d2},
\]
this yields
\begin{equation}
\label{source_term_estimate}
(1+t)^d
\|\nabla v(t)\|_{L^\infty}
\|u(t)\|_{L^2}^2
\lesssim
(1+t)^{-\frac12}.
\end{equation}
Thus
\begin{equation}
\label{Energy_w_4}
F'(t)
\lesssim
(1+t)^{d-1}
\int_{S(t)}|\widehat w(\xi,t)|^2\,d\xi
+
(1+t)^{-\frac12}.
\end{equation}
It remains to estimate the low-frequency part of \(\widehat w\). Taking the
Fourier transform of \eqref{w_Equation}, we obtain
\[
\partial_t\widehat w(\xi,t)
+
m(\xi)\widehat w(\xi,t)
=
-\widehat{\nabla\cdot f(u)}(\xi,t),
\qquad
\widehat w(\xi,0)=0.
\]
Hence
\begin{equation}
\label{w_Duhamel}
\widehat w(\xi,t)
=
-\int_0^t
e^{-m(\xi)(t-s)}
\widehat{\nabla\cdot f(u)}(\xi,s)\,ds.
\end{equation}
Since \(f(u)=u^2\), Young's convolution inequality gives
\[
|\widehat{\nabla\cdot f(u)}(\xi,s)|
\lesssim
|\xi|\,|\widehat u(\cdot,s)\ast \widehat u(\cdot,s)|(\xi)
\lesssim
|\xi|\,\|u(s)\|_{L^2}^2.
\]
Using the decay estimate for \(u\), we obtain
\begin{equation}
\label{Fourier_source_bound}
|\widehat{\nabla\cdot f(u)}(\xi,s)|
\lesssim
|\xi|(1+s)^{-\frac d2}.
\end{equation}
For \(\xi\in S(t)\), \eqref{volume_S_t} gives
\[
|\xi|\lesssim (1+t)^{-1/2}.
\]
Therefore, by \eqref{w_Duhamel} and \eqref{Fourier_source_bound},
\[
|\widehat w(\xi,t)|
\lesssim
(1+t)^{-1/2}
\int_0^t (1+s)^{-\frac d2}\,ds.
\]
Since \(d\geq 2\), for every \(\varepsilon>0\),
\[
\int_0^t (1+s)^{-\frac d2}\,ds
\lesssim
(1+t)^\varepsilon.
\]
Consequently,
\begin{equation}
\label{low_frequency_w_hat_bound}
|\widehat w(\xi,t)|
\lesssim 
(1+t)^{-\frac12+\varepsilon},
\qquad
\xi\in S(t).
\end{equation}
Combining \eqref{low_frequency_w_hat_bound} with \eqref{volume_S_t}, we get
\[
\int_{S(t)}|\widehat w(\xi,t)|^2\,d\xi
\lesssim 
(1+t)^{-1+2\varepsilon}|S(t)|
\lesssim
(1+t)^{-1-\frac d2+2\varepsilon}.
\]
Therefore, from \eqref{Energy_w_4},
\[
F'(t)
\lesssim 
(1+t)^{d-1}
(1+t)^{-1-\frac d2+2\varepsilon}
+
(1+t)^{-\frac12}.
\]
Hence
\begin{equation}
\label{F_prime_estimate}
F'(t)
\lesssim 
(1+t)^{\frac d2-2+2\varepsilon}
+
(1+t)^{-\frac12}.
\end{equation}
Integrating \eqref{F_prime_estimate} from $0$ to \(t\), we obtain
\[
F(t)
\lesssim 
1
+
(1+t)^{\frac d2-1+2\varepsilon}
+
(1+t)^{\frac12}.
\]
Since \(F(t)=(1+t)^d\|w(t)\|_{L^2}^2\), this implies
\[
\|w(t)\|_{L^2}^2
\lesssim 
(1+t)^{-d}
+
(1+t)^{-\frac d2-1+2\varepsilon}
+
(1+t)^{-d+\frac12}.
\]
For \(d\geq 2\), this gives
\begin{equation}
\label{w_decay}
\|w(t)\|_{L^2}
\lesssim 
(1+t)^{-\frac d4-\frac14+\varepsilon}.
\end{equation}
Choosing \(0<\varepsilon<1/4\), we obtain a strictly faster decay than
\((1+t)^{-d/4}\).

Finally, combining \eqref{Lower_Bound_v_NL} and \eqref{w_decay}, we have
\[
\|u(t)\|_{L^2}
\geq
\|v(t)\|_{L^2}-\|w(t)\|_{L^2}
\geq
c_\eta (1+t)^{-\frac d4},
\]
for all \(t\geq T_\eta\), with \(T_\eta\) large enough. This proves the
lemma.
\end{proof}

\begin{remark}
The assumption
\[
|\widehat u_0(\xi)|\geq \eta,
\qquad |\xi|\leq \Lambda_2,
\]
is implied by \(u_0\in L^1(\mathbb R^d)\) and
\[
\int_{\mathbb R^d}u_0(x)\,dx=\widehat u_0(0)\neq 0,
\]
after possibly decreasing \(\eta\) and \(\Lambda_2\), since \(\widehat u_0\) is
continuous near the origin. Under this assumption, similar optimality results
for the Navier--Stokes equations were proved in
\cite{Schonbek_1986,Schonb_1991}.
\end{remark}

\section{Faster time-decay estimates}\label{Section_Improved_Decay}
 In this section, we show that it is possible to obtain a faster decay rate for initial data $u_0 \in \dot{B}^{-d/2-1}_{2,\infty}(\mathbb{R}^d)$ that satisfies $\int_{\R^d}u_0(x)dx=0$.  This condition provides a low-frequency cancellation and yields an improved decay rate for the linearized problem. We then show that this faster decay rate extends to the nonlinear problem thanks to the divergence structure of the nonlinear term.
\subsection{Faster decay estimates for the linear system}
In this section, we show that it is possible to obtain a better decay rate in the linear problem for a special class of initial data. The key idea is that for sufficiently small frequencies, and if $\hat{u}_0(0)=0$, then we have $|\hat{u}_0(\xi)|\lesssim |\xi|$. This low frequency cancellation provides an extra factor of $|\xi|$ compared to the case $\hat{u}(0)\neq 0$, resulting 
 in  an improved  decay rate of the solution by a factor $t^{-1/2}$. 
\begin{lemma}\label{Lemma_weighted}
Assume that 
\[
u_0\in L^2,\quad|x|u_0 \in L^1(\mathbb{R}^d)
\quad \text{and} \quad
\int_{\mathbb{R}^d} u_0(x)\,dx = 0,
\]
then
\[
u_0 \in \dot{B}^{-d/2-1}_{2,\infty}(\mathbb{R}^d).
\]
\end{lemma}
\begin{proof}
The condition $|x|u_0 \in L^1(\mathbb{R}^d)$ implies that $\widehat{u}_0$ is Lipschitz continuous near the origin, namely
\[
|\widehat{u}_0(\xi)-\widehat{u}_0(0)| \lesssim |\xi|\,\||x|u_0\|_{L^1}.
\]
If in addition $\int_{\mathbb{R}^d} u_0(x)\,dx=0$, then $\widehat{u}_0(0)=0$ and thus
\begin{equation}\label{Low_Estimate}
|\widehat{u}_0(\xi)| \lesssim |\xi| \quad \text{for } |\xi|\ll 1.
\end{equation}  
Using Plancherel's theorem together with \eqref{Low_Estimate}, we have for $|\xi|\approx 2^q\leq 1$
\begin{equation}
\|\Delta_q u_0\|_{L^2}
\lesssim
\||x|u_0\|_{L^1}\left(
\int_{|\xi|\lesssim 2^q}
|\xi|^2
\,d\xi
\right)^{1/2}\lesssim 2^{\frac{q(d+2)}{2}}.
\end{equation}
Hence, 
\begin{equation}
\sup_{q\leq 0} 2^{-q(\frac{d}{2}+1)}\|\Delta_q u_0\|_{L^2}<\infty.
\end{equation}
For $q>0,$ one uses that $u_0\in L^2$:
\begin{equation}
 2^{-q(\frac{d}{2}+1)}\|\Delta_q u_0\|\leq \|u_0\|_{L^2}.
\end{equation}

This yields the result that 
\[
u_0 \in \dot B^{-d/2-1}_{2,\infty}(\mathbb{R}^d),
\]
and consequently
\[
\nabla u_0 \in \dot B^{-d/2}_{2,\infty}(\mathbb{R}^d).
\]
\end{proof}

We now justify an improved decay rate for the linearized problem \eqref{Main_problem_Linear} which can be deduced  from Theorem \ref{Theorem_Main_Linear}. 
\begin{proposition}
Let $\sigma>-\frac{d}{2}-1$. Assume that  $u_0\in \dot{B}^{-\frac{d}{2}-1}_{2,\infty}\cap \dot{B}%
^{\sigma}_{2,r}$.
The following decay estimate holds: 
\begin{equation}  \label{Main_Decay_estimate_improved}
\Vert u(t)\Vert_{\dot{B}^{\sigma}_{2,r}}\leq C (1+t)^{-\frac{\sigma+d/2+1}{2}%
}\Vert u_0 \Vert_{\dot{B}_{2,\infty}^{-\frac{d}{2}-1}}+Ce^{-\lambda t}\Vert u_0 \Vert_{%
\dot{B}_{2,r}^{\sigma}},
\end{equation}
where $\lambda$ and $C$ are positive constants.
\end{proposition}
\begin{remark}
From \eqref{Estimate_L_p} and by choosing $s=d/2+1$, we have 
\begin{equation}  \label{Estimate_L_p_improved}
\begin{aligned}
\Vert u(t)\Vert_{L^p}
\leq C(\Vert u_0 \Vert_{\dot{B}_{2,\infty}^{-d/2-1}}+\Vert u_0 \Vert_{\dot{B}%
_{2,1}^{\sigma}}) (1+t)^{-\frac{d}{2}(1-1/p)-1/2}.
\end{aligned}  
\end{equation}
In particular, for $p=2$, this gives 
\begin{equation}  \label{Estimate_L_2_improved}
\begin{aligned}
\Vert u(t)\Vert_{L^2}
\leq C(\Vert u_0 \Vert_{\dot{B}_{2,\infty}^{-d/2-1}}+\Vert u_0 \Vert_{\dot{B}%
_{2,1}^{\sigma}}) (1+t)^{-\frac{d}{4}-\frac{1}{2}}.
\end{aligned}
\end{equation}
Both estimates \eqref{Estimate_L_2_improved} and \eqref{Estimate_L_p_improved} give an improved decay rate by a factor of $(1+t)^{-1/2}$.     
\end{remark}

\subsection{Faster decay estimates for the nonlinear problem}
\label{Subsection_Faster_Decay_Nonlinear}

Extending \eqref{Main_Decay_estimate_improved} to the nonlinear problem is more delicate than in the linear case. Indeed, the propagation of the negative Besov norm
\(\dot B^{-d/2-1}_{2,\infty}\) is not automatic for the nonlinear equation. We must rely on the divergence structure of the nonlinear term. The low-frequency part of the Duhamel term is controlled by the \(L^2\)-decay estimate \eqref{Estimate_Decay_Main}, while the high-frequency part is controlled by the decay estimates of Theorem \ref{Theorem_Decay_2}.

\begin{proposition}
\label{Prop_Nonlinear_Improved_Decay}
Let \(d\geq 3\), \(1\leq r\leq \infty\), and let $-\frac d2-1<\sigma\leq \frac d2-1$.
Set
\[
a_\sigma:=\frac{\sigma+d/2+1}{2}.
\]
Let \(u\) be the global solution of \eqref{Main_problem} obtained in Theorem
\ref{Main_Theorem_Nonl}. Assume that
\[
u_0\in
\dot B^{-d/2-1}_{2,\infty}
\] and that the assumptions of Theorem \ref{theorem_Decay}
and Theorem \ref{Theorem_Decay_2} are satisfied with \(\sigma_1=d/2\). Then
\begin{equation}
\label{Nonlinear_Improved_Decay_Estimate}
\|u(t)\|_{\dot B^\sigma_{2,r}}
\lesssim 
(1+t)^{-\frac{\sigma+d/2+1}{2}}.
\end{equation}
\end{proposition}

\begin{proof}
By Duhamel's formula, we have
\begin{equation}
\label{Duhamel_Nonlinear_Improved}
u(t)
=
e^{t\Delta P}u_0
-
\int_0^t e^{(t-s)\Delta P}\nabla\cdot f(u(s))\,ds.
\end{equation}
Using the linear estimate \eqref{Main_Decay_estimate_improved}, we obtain
\begin{equation}
\label{Duhamel_Estimate_Nonlinear_Improved}
\begin{aligned}
\|u(t)\|_{\dot B^\sigma_{2,r}}
\lesssim{}&
(1+t)^{-a_\sigma}
\left(
\|u_0\|_{\dot B^{-d/2-1}_{2,\infty}}
+
\|u_0\|_{\dot B^\sigma_{2,r}}
\right)
\\
&+
\int_0^t
(1+t-s)^{-a_\sigma}
\|\nabla\cdot f(u(s))\|_{\dot B^{-d/2-1}_{2,\infty}}
\,ds
\\
&+
\int_0^t
e^{-\lambda(t-s)}
\|\nabla\cdot f(u(s))\|_{\dot B^\sigma_{2,r}}
\,ds .
\end{aligned}
\end{equation}
We have
\begin{equation}
\label{Duhamel_Estimate_Nonlinear_Improved_2}
\begin{aligned}
\|u(t)\|_{\dot B^\sigma_{2,r}}
\lesssim{}&
(1+t)^{-a_\sigma}
\left(
\|u_0\|_{\dot B^{-d/2-1}_{2,\infty}}
+
\|u_0\|_{\dot B^\sigma_{2,r}}
\right)
\\
&+
\int_0^t
(1+t-s)^{-a_\sigma}
\|f(u(s))\|_{\dot B^{-d/2}_{2,\infty}}
\,ds
\\
&+
\int_0^t
e^{-\lambda(t-s)}
\|f(u(s))\|_{\dot B^{\sigma+1,H}_{2,r}}
\,ds .
\end{aligned}
\end{equation}

We first estimate the low-frequency contribution. Since
\[
L^1(\mathbb R^d)\hookrightarrow \dot B^{-d/2}_{2,\infty}(\mathbb R^d),
\]
and since \(f(0)=f'(0)=0\), we have
\begin{equation}
\label{Low_Frequency_Source_Estimate}
\|f(u(t))\|_{\dot B^{-d/2}_{2,\infty}}
\lesssim
\|f(u(t))\|_{L^1}
\lesssim
\|u(t)\|_{L^2}^2.
\end{equation}
Using \eqref{Estimate_Decay_Main}, we infer that
\begin{equation}
\label{Low_Frequency_Source_Decay}
\|f(u(t))\|_{\dot B^{-d/2}_{2,\infty}}
\lesssim
(1+t)^{-d/2}.
\end{equation}
Hence
\begin{equation}
\label{Low_Frequency_Duhamel_Term}
\begin{aligned}
&\int_0^t
(1+t-s)^{-a_\sigma}
\|f(u(s))\|_{\dot B^{-d/2}_{2,\infty}}
\,ds\lesssim
\int_0^t
(1+t-s)^{-a_\sigma}
(1+s)^{-d/2}
\,ds .
\end{aligned}
\end{equation}
Because \(d\geq 3\), we have \(d/2>1\). Moreover, since
\[
\sigma\leq \frac d2-1,
\]
we have
\[
a_\sigma=\frac{\sigma+d/2+1}{2}\leq \frac d2.
\]
Therefore, by the standard convolution estimate
\[
\int_0^t
(1+t-s)^{-a}
(1+s)^{-b}
\,ds
\lesssim
(1+t)^{-\min\{a,b\}},
\qquad
a,b>0,\quad \max\{a,b\}>1,
\]
we obtain
\begin{equation}
\label{Low_Frequency_Duhamel_Term_Final}
\int_0^t
(1+t-s)^{-a_\sigma}
\|f(u(s))\|_{\dot B^{-d/2}_{2,\infty}}
\,ds
\lesssim
(1+t)^{-a_\sigma}.
\end{equation}
We now estimate the high-frequency contribution. We use the product estimates established in the proof of Theorem~\ref{Main_Theorem_Nonl}. Since \(f(0)=f'(0)=0\), the nonlinear term is at least quadratic, and we have
\begin{equation}
\label{HF_Product_Estimate}
\|f(u(t))\|_{\dot B^{\sigma+1,H}_{2,r}}
\lesssim
\|u(t)\|_{\mathcal{X}}
\|u(t)\|_{\dot B^{d/2+1}_{2,1}}.
\end{equation}
Since
\[
\|u(t)\|_{\dot B^{d/2+1}_{2,1}}
\lesssim
\|u(t)\|_{\mathcal{X}},
\]
it follows that
\begin{equation}
\label{HF_Product_Estimate_2}
\|f(u(t))\|_{\dot B^{\sigma+1,H}_{2,r}}
\lesssim
\|u(t)\|_{\mathcal{X}}^2.
\end{equation}

Taking \(\sigma_1=d/2\) in Theorem~\ref{Theorem_Decay_2}, we obtain
\[
\alpha_1
=
\frac12\left(\frac d2+\frac d2-1\right)
=
\frac{d-1}{2}.
\]
Therefore, \eqref{Decay_Rate_Besov} yields
\[
\|u(t)\|_{\mathcal{X}}
\lesssim
(1+c_0t)^{-\frac{d-1}{2}}.
\]
Combining this estimate with \eqref{HF_Product_Estimate_2}, we get
\begin{equation}
\label{HF_Source_Decay}
\|f(u(t))\|_{\dot B^{\sigma+1,H}_{2,r}}
\lesssim
(1+c_0t)^{-(d-1)}.
\end{equation}
Since \(d\geq 3\) and \(\sigma\leq d/2-1\), we have
\[
a_\sigma\leq \frac d2\leq d-1.
\]
Hence,
\begin{equation}
\label{HF_Source_Decay_2}
\|f(u(t))\|_{\dot B^{\sigma+1}_{2,r}}
\lesssim
(1+c_0t)^{-a_\sigma}.
\end{equation}
Consequently,
\begin{equation}
\label{HF_Duhamel_Term}
\begin{aligned}
\int_0^t
e^{-\lambda(t-s)}
\|f(u(s))\|_{\dot B^{\sigma+1,H}_{2,r}}
\,ds
&\lesssim
\int_0^t
e^{-\lambda(t-s)}
(1+c_0s)^{-a_\sigma}
\,ds
\\
&\lesssim
(1+c_0t)^{-a_\sigma}.
\end{aligned}
\end{equation}
Combining \eqref{Duhamel_Estimate_Nonlinear_Improved_2},
\eqref{Low_Frequency_Duhamel_Term_Final}, and
\eqref{HF_Duhamel_Term}, we finally obtain
\[
\|u(t)\|_{\dot B^\sigma_{2,r}}
\lesssim
(1+c_0t)^{-a_\sigma}.
\]
This concludes the proof of the proposition.
\end{proof}

\section{A local existence theory}\label{Section_Local_Existence}
Here we follow a standard local existence scheme, as presented for instance in \cite[Chapter 4]{Bahouri_2011_1}.
\begin{theorem}
\label{them1:local}
Assume that $u_0$  is in  $ \dot{B}^{\frac{d}{2}-1}_{2,1}\cap \dot{B}^{\frac{d}{2}+1}_{2,1}.$ There exists a time $T>0$ such that \eqref{Main_problem} supplemented with the initial datum $u_0$ admits a unique solution $u$ satisfying
$$u \in \mathcal{C}([0,T); \dot B^{\frac{d}{2}-1}_{2,1} \cap \dot B^{\frac{d}{2}+1}_{2,1})\cap L^1([0,T); \dot B^{\frac{d}{2}+1}_{2,1}).
$$

\end{theorem}

\begin{proof}
  With the a priori estimates obtained in Section \ref{Section_Nonl}, the local well-posedness can be established by a standard Friedrichs approximation scheme combined with a linearized Picard iteration. 
The result follows from a standard Friedrichs approximation procedure; see, for
instance, \cite[Chapter~4]{Bahouri_2011_1}. More precisely, applying a spectral truncation, we obtain a sequence of finite-dimensional approximate
problems, which admit smooth solutions on a maximal time interval. The estimates
established in Section 4 apply uniformly to those approximate solutions. By choosing
$T>0$ sufficiently small, so that the dissipative norm of the solution to the
corresponding linear problem is small on $[0,T]$, a standard bootstrap argument yields
uniform bounds in
\[
\widetilde L^\infty_T
 \bigl(\dot B^{d/2-1,L}_{2,1}
       \cap \dot B^{d/2+1,H}_{2,1}\bigr)
\cap
L^1_T\bigl(\dot B^{d/2+1}_{2,1}\bigr).
\]
where the notation $\widetilde L^\infty_T$ corresponds to a Chemin-Lerner space, see \cite[Sec. 2.6.3]{Bahouri_2011_1}.
Compactness arguments then allow us to pass to the limit and obtain a solution of
\eqref{Main_problem}. Time continuity and uniqueness follow from the analogous
estimates applied to the difference of two solutions in a lower regularity norm. As all these arguments are
classical, we omit the remaining details.


\end{proof}

\bigbreak
\bigbreak

 \subsection*{Acknowledgments}

T. Crin-Barat is supported by the project ANR-24-CE40-3260 – Hyperbolic Equations, Approximations $\&$ Dynamics (HEAD) and the project ANR-25-CE40-5565 (Cookie).

\subsection*{Data availability statement}

Data sharing is not applicable to this article, as no datasets were generated or analyzed during the current study.

\subsection*{Conflict of interest statement}

The authors declare that they have no conflict of interest.

\appendices 
\section{Appendix}\label{Appendix_A}
In this section, we collect a few technical lemmas that have been used in the proof.
\begin{lemma}[\cite{CRINBARAT20221}]\label{Lemma_Diff_Ineq} Let $X:[0,T]\rightarrow \R_+$ be a continuous function such that $X^2$ is differentiable. Assume that there exists a constant $B\geq 0$ and a measurable function $A:[0,T]\rightarrow\R_+$ such that  
\begin{equation}
\frac{1}{2}\frac{\textup{d}}{\dt} X^2+BX^2\leq AX,\qquad a.e. \quad \text{on}\quad [0,T].
\end{equation}
Then, for all $t\in[0,T]$, we have 
\begin{equation}
X(t)+B\int_0^t X(s)\ds\leq X(0)+\int_0^t A(s)\ds. 
\end{equation}
\end{lemma}

We also recall the following classical commutator estimate (see \cite[Chapter 2]{Bahouri_2011_1}).
\begin{lemma}\label{Commutator_estimate}
For $s\in (-\frac{d}{2}, \frac{d}{2}+1]$, for $a,u\in\mathcal{S}'(\R^d)$, we have 
\begin{equation}
\|[a,\dot{\Delta}_{k}] \nabla u \|_{L^2}\leq Cc_k2^{-ks}\|\nabla a\|_{\dot{B}%
_{2,1}^{d/2}}\|u\|_{\dot{B}%
_{2,1}^{s}}
\end{equation}
with $\sum_{k\in \Z}c_k=1$. 
\end{lemma}
We recall the following composition estimate (see \cite{Danchin:2017aa}).
\begin{lemma}\label{Composition_classical}
Let $F : \mathbb{R} \to \mathbb{R}$ be smooth with $F(0)=0$. For all $1 \le p, r \le \infty$ and $\sigma>0$ we have 
$
F(f) \in \dot{B}^{\sigma}_{p,r} \cap L^\infty \quad \text{for } f \in \dot{B}^{\sigma}_{p,r} \cap L^\infty,
$
and
\begin{equation}
\|F(f)\|_{\dot{B}^{\sigma}_{p,r}} \leq C(\|f\|_{L^\infty})\|f\|_{\dot{B}^{\sigma}_{p,r}}
\end{equation}
with $C$ depending only on $\|f\|_{L^\infty}$, $F'$ (and higher derivatives), $\sigma$, $p$, and $d$.
\end{lemma}
The following lemma about low-frequency composition estimates has been proved in \cite{Crin-Barat:2025aa}.
\begin{lemma} \label{Lemma_composition_L}
Let $s>0,\, \sigma\in \R$. Then, for any smooth function $F(u)$ satisfying $F(0)=F'(0)=0$, there is a constant $C>0$ depending only on $\|u\|_{L^\infty},\, s,\, \sigma$ and $d$ such that 
\begin{equation}
\Vert F(u)\Vert _{\dot{B}%
_{2,1}^{s,L}}\leq C (\|u\|_{\dot{B}%
_{2,1}^{\frac{d}{2},L}}+\|u\|_{\dot{B}%
_{2,1}^{\frac{d}{2},H}})\|u\|_{\dot{B}%
_{2,1}^{s,L}}
+C(\|u\|_{\dot{B}%
_{2,1}^{\frac{d}{2}-1,L}}+\|u\|_{\dot{B}%
_{2,1}^{\frac{d}{2},H}})\|u\|_{\dot{B}%
_{2,1}^{\sigma,H}}
\end{equation}

\end{lemma}

%

We also recall the following estimates from 
\cite[Chapter 2]{Bahouri_2011_1}.

\begin{lemma}\label{Commutator}
  The following inequalities hold true
  \begin{enumerate}
\item If  $-d/2\leq s\leq d/2+1$, then 
\begin{equation}\label{Commu_s}
\sup_{k\in \Z} 2^{ks}\Vert [w, \dot{\Delta}_{k} ]\nabla v\Vert
_{L^{2}}\leq C \|\nabla w\|_{\dot{B}%
_{2,1}^{\frac{d}{2}}}\|v\|_{\dot{B}%
_{2,\infty}^{s}}
\end{equation}
\item If $-d/2\leq \sigma\leq d/2$, then
\begin{equation}\label{Product_sigma}
\|fg\|_{\dot{B}%
_{2,1}^{-\sigma}}\leq C\|f\|_{\dot{B}%
_{2,1}^{\frac{d}{2}}} \|g\|_{\dot{B}%
_{2,1}^{-\sigma}}.
\end{equation}

\end{enumerate}

\end{lemma}


\begin{thebibliography}{10}

\bibitem{Bahouri_2011_1}
H.~Bahouri, J-Y Chemin, and R.~Danchin.
\newblock {\em Fourier analysis and nonlinear partial differential equations},
  volume 343.
\newblock Springer Science \& Business Media, 2011.

\bibitem{Capiro_1996_2}
A.~Carpio.
\newblock Large-time behavior in incompressible {N}avier-{S}tokes equations.
\newblock {\em SIAM J. Math. Anal.}, 27(2):449--475, 1996.

\bibitem{Capiro_1996_1}
A.~Carpio.
\newblock Large time behaviour in convection-diffusion equations.
\newblock {\em Ann. Scuola Norm. Sup. Pisa Cl. Sci. (4)}, 23(3):551--574, 1996.

\bibitem{Chemin_Lerner_1995}
J-Y Chemin and N.~Lerner.
\newblock Flot de champs de vecteurs non lipschitziens et {\'e}quations de
  {N}avier--{S}tokes.
\newblock {\em J. Differential Equations}, 121(2):314--328, 1995.

\bibitem{Crin-Bara_Danchin_2021}
T.~Crin-Barat and R.~Danchin.
\newblock Partially dissipative one-dimensional hyperbolic systems in the
  critical regularity setting, and applications,.
\newblock {\em Pure Appl. Math}, 2021.

\bibitem{CRINBARAT20221}
T.~Crin-Barat and R.~Danchin.
\newblock Partially dissipative hyperbolic systems in the critical regularity
  setting: The multi-dimensional case.
\newblock {\em Journal de Math{\'e}matiques Pures et Appliqu{\'e}es},
  165:1--41, 2022.

\bibitem{Crin-Barat:2025aa}
T.~Crin-Barat, L.-Y. Shou, and J.~Zhang.
\newblock Strong relaxation limit and uniform time asymptotics of the jin-xin
  model in the lp framework.
\newblock {\em Science China Mathematics}, 2025.

\bibitem{Zhao_Zhu_2015}
C.Zhao and H.~Zhu.
\newblock Upper bound of decay rate for solutions to the
  {N}avier--{S}tokes--{V}oigt equations in ${R}^3$.
\newblock {\em Appl. Math. Comput}, 256:183--191, 2015.

\bibitem{Danchin_2001_2}
R.~Danchin.
\newblock Global existence in critical spaces for flows of compressible viscous
  and heat-conductive gases.
\newblock {\em Arch. Ration. Mech. Anal.}, 160(1):1--39, 2001.

\bibitem{Danchin_EMS}
R.~Danchin.
\newblock Partially dissipative systems in the critical regularity setting, and
  strong relaxation limit.
\newblock {\em EMS Surv. Math. Sci.}, 9(1):135--192,, 2022.

\bibitem{Danchin:2017aa}
R.~Danchin and J.~Xu.
\newblock Optimal time-decay estimates for the compressible {N}avier--{S}tokes
  equations in the critical ${L}^p$ framework.
\newblock {\em Archive for Rational Mechanics and Analysis}, 224(1):53--90,
  2017.

\bibitem{Francesco:2007aa}
M.~di~Francesco.
\newblock Initial value problem and relaxation limits of the hamer model for
  radiating gases in several space variables.
\newblock {\em Nonlinear Differential Equations and Applications NoDEA},
  13(5):531--562, 2007.

\bibitem{Duan_Klem_Zhu_2010}
R.~Duan, K.~Fellner, and C.~Zhu.
\newblock Energy method for multi-dimensional balance laws with non-local
  dissipation.
\newblock {\em J. Math. Pures Appl. (9)}, 93(6):572--598, 2010.

\bibitem{Duan_Ruan_Zhu_2012}
R.~Duan, L.~Ruan, and C.~Zhu.
\newblock Optimal decay rates to conservation laws with diffusion-type terms of
  regularity-gain and regularity-loss.
\newblock {\em Math. Mod. Meth. Appl. Sci.}, 22(7), 2012.

\bibitem{Duro_Carpio_2001}
G.~Duro and A.~Carpio.
\newblock Asymptotic profiles for convection-diffusion equations with variable
  diffusion.
\newblock {\em Nonlinear Anal.}, 45(4, Ser. A: Theory Methods):407--433, 2001.

\bibitem{Duro_Zuazua_1999}
G.~Duro and E.~Zuazua.
\newblock Large time behavior for convection-diffusion equations in {${\bf
  R}^N$} with asymptotically constant diffusion.
\newblock {\em Comm. Partial Differential Equations}, 24(7-8):1283--1340, 1999.

\bibitem{Escob_Vazq_Zuazua_1993_2}
M.~Escobedo, J.~L V{{\'a}}zquez, and E.~Zuazua.
\newblock Asymptotic behaviour and source-type solutions for a
  diffusion-convection equation.
\newblock {\em Arch. Rational Mech. Anal.}, 124(1):43--65, 1993.

\bibitem{EscobVasqZuazua_1993}
M.~Escobedo, J.~L. V{{\'a}}zquez, and E.~Zuazua.
\newblock A diffusion-convection equation in several space dimensions.
\newblock {\em Indiana Univ. Math. J.}, 42(4):1413--1440, 1993.

\bibitem{EscoZuazua_1991}
M.~Escobedo and E.~Zuazua.
\newblock Large time behavior for convection-diffusion equations in {${\bf
  R}^N$}.
\newblock {\em J. Funct. Anal.}, 100(1):119--161, 1991.

\bibitem{G_Gui_2013}
G.~Gui.
\newblock {\em Stability to the Incompressible {N}avier-{S}tokes Equations},
  volume 132.
\newblock Spriger, Heidelberg, 2013.

\bibitem{Guo_Wang_2012}
Y.~Guo and Y.~Wang.
\newblock Decay of dissipative equations and negative {S}obolev spaces.
\newblock {\em Communications in Partial Differential Equations},
  37(12):2165--2208, 2012.

\bibitem{Ha71}
K.~Hamer.
\newblock Nonlinear effects on the propagation of sound waves in a radiating
  gas.
\newblock {\em Quart. J. Mech. Appl. Math.}, 24:155--168, 1971.

\bibitem{Iguchi_2002}
T.~Iguchi and S.~Kawashima.
\newblock On space-time decay properties of solutions to hyperbolic-elliptic
  coupled systems.
\newblock {\em Hiroshima Mathematical Journal}, 32(2):229--308, 2002.

\bibitem{Karch_Schon_2002}
G.~Karch and M.~E. Schonbek.
\newblock On zero mass solutions of viscous conservation laws.
\newblock {\em Comm. Partial Differential Equations}, 27(9-10):2071--2100,
  2002.

\bibitem{Kawashima_1998}
S.~Kawashima and S.~Nishibata.
\newblock Weak solutions with a shock to a model system of the radiating gas.
\newblock {\em Sci. Bull. Josai Univ}, 5:119--130, 1998.

\bibitem{Kawash_Nishibata_1999}
S.~Kawashima and S.~Nishibata.
\newblock Shock waves for a model system of the radiating gas.
\newblock {\em SIAM J. Math. Anal.}, 30(1):95--117 (electronic), 1999.

\bibitem{Lattanzio_2003}
C.~Lattanzio and P.~Marcati.
\newblock Global well-posedness and relaxation limits of a model for radiating
  gas.
\newblock {\em Journal of Differential Equations}, 190(2):439--465, 2003.

\bibitem{Lattanzio_2009}
C.~Lattanzio, C.~Mascia, T.~Nguyen, R.G Plaza, and K.~Zumbrun.
\newblock Stability of scalar radiative shock profiles.
\newblock {\em SIAM Journal on Mathematical Analysis}, 41(6):2165--2206, 2009.

\bibitem{Lattanzio_2007}
C.~Lattanzio, C.and~Mascia and D.~Serre.
\newblock Shock waves for radiative hyperbolic--elliptic systems.
\newblock {\em Indiana University Mathematics Journal}, pages 2601--2640, 2007.

\bibitem{Lauren_2005}
P.~Lauren{\c{c}}ot.
\newblock Asymptotic self-similarity for a simplified model for radiating
  gases.
\newblock {\em Asymptot. Anal.}, 42(3-4):251--262, 2005.

\bibitem{Lin_2007}
C.~Lin, J-F Coulombel, and T.~Goudon.
\newblock Asymptotic stability of shock profiles in radiative hydrodynamics.
\newblock {\em Comptes Rendus Mathematique}, 345(11):625--628, 2007.

\bibitem{Ma76}
A.~Matsumura.
\newblock On the asymptotic behavior of solutions of semi-linear wave
  equations.
\newblock {\em Publ. Res. Inst. Math. Sci. Kyoto. Univ}, 12(1):169--189, 1976.

\bibitem{Nash_1}
J.~Nash.
\newblock Continuity of solutions of parabolic and elliptic equations.
\newblock {\em American Journal of Mathematics}, 80(4):931--954, 1958.

\bibitem{Rosenau_1989}
P.~Rosenau.
\newblock Extending hydrodynamics via the regularization of the chapman-enskog
  expansion.
\newblock {\em Physical Review A}, 40(12):7193, 1989.

\bibitem{Schonbek_1986}
M.~E. Schonbek.
\newblock Large time behaviour of solutions to the {N}avier-{S}tokes equations.
\newblock {\em Communications in Partial Differential Equations},
  11(7):733--763, (1986).

\bibitem{Schonb_1991}
M.~E. Schonbek.
\newblock Lower bounds of rates of decay for solutions to the {N}avier-{S}tokes
  equations.
\newblock {\em J. Amer. Math. Soc.}, 4(3):423--449, 1991.

\bibitem{Serre_2003}
D.~Serre.
\newblock ${L}^1$-stability of constants in a model for radiating gases.
\newblock {\em Communications in Mathematical Sciences}, 1(1):197--205, 2003.

\bibitem{Sohinger_Strain_2014}
V.~Sohinger and R.~M. Strain.
\newblock The {B}oltzmann equation, {B}esov spaces, and optimal time decay
  rates in $ \mathbb{R}_x^n$.
\newblock {\em Advances in Mathematics}, 261:274--332, 2014.

\bibitem{Stein_1}
E.~M. Stein.
\newblock {\em Singular integrals and differentiability properties of functions
  (PMS-30)}, volume~30.
\newblock Princeton university press, 2016.

\bibitem{VINCENTI}
W.~G. Vincent and C.~H. Kruger.
\newblock {\em Introduction to Physical Gas Dynamics,}.
\newblock A Wiley and Sons , New York, 1965.

\bibitem{XuXin}
Z.~Xin and J.~Xu.
\newblock Optimal decay for the compressible {N}avier-{S}tokes equations
  without additional smallness assumptions.
\newblock {\em Journal of Differential Equations}, 274, 543-575, 2021.

\bibitem{Xu_Kawashima_2015}
J.~Xu and S.~Kawashima.
\newblock The optimal decay estimates on the framework of {B}esov spaces for
  generally dissipative systems.
\newblock {\em Arch. Ration. Mech. Anal.}, 218:275--315, 2015.

\bibitem{Kawashima_1}
J.~Xu, N.~Mori, and S.~Kawashima.
\newblock Global existence and minimal decay regularity for the {T}imoshenko
  system: The case of non-equal wave speeds.
\newblock {\em Journal of Differential Equations}, 259(11):5533--5553, 2015.

\bibitem{Zha_JMP}
J.~Zhao.
\newblock The optimal temporal decay estimates for the fractional power
  dissipative equation in negative besov spaces.
\newblock {\em J. Math. Phys.}, 57(5):051504, 2016.

\bibitem{Zuazua_1993}
E.~Zuazua.
\newblock Weakly nonlinear large time behavior in scalar convection-diffusion
  equations.
\newblock {\em Differential Integral Equations}, 6(6):1481--1491, 1993.

\end{thebibliography}

\end{document}